\documentclass[12pt]{article}
\usepackage{amsfonts}
\usepackage{amssymb}
\usepackage{mathrsfs}
\usepackage{amsmath}
\usepackage{amsthm}
\usepackage{amstext}
\usepackage{amsopn}
\usepackage{graphicx}
\usepackage{subcaption}
\usepackage{color}
\newtheorem{theorem}{Theorem}[section]
\newtheorem{lemma}[theorem]{Lemma}
\newtheorem{proposition}[theorem]{Proposition}

\theoremstyle{definition}

\theoremstyle{remark}
\newtheorem{remark}[theorem]{Remark}

\numberwithin{equation}{section}

\newcommand{\ba}{\begin{array}}
\newcommand{\ea}{\end{array}}
\newcommand{\f}{\frac}

\newcommand{\Om}{\Omega}

\newcommand{\la}{\lambda}

\newcommand{\ds}{\displaystyle}

\begin{document}
\date{}
\title{ \bf\large{Hopf bifurcations  in a reaction-diffusion model with a
general advection term and delay effect}}
\author{
Jingxiao Song, Chengwei Ren\; and\; Shaofen Zou\footnote{Corresponding author,
Email: shaofenzzou@hnu.edu.cn (S. Zou).
}\\[-1mm]
{\small School of Mathematics, Hunan University}\\[-2mm]
{\small Changsha, Hunan 410082, PR China}\\[1mm]
}
\maketitle

\begin{abstract}
{
This paper investigates a class of reaction-diffusion population models defined on a bounded domain, characterized by a general time-delayed per capita growth rate and a general advection term. Notably, the growth rate encompasses both Logistic-type and weak Allee effect-type dynamical behaviors. By applying the Lyapunov method, we establish the existence of spatially inhomogeneous steady states when a parameter approaches the principal eigenvalue of a non-self-adjoint elliptic operator. A detailed analysis of the characteristic equation further confirms the existence of Hopf bifurcations originating from these steady states. Subsequently, by applying center manifold reduction and normal form theory, we ascertain the direction of these Hopf bifurcations and the stability of the resulting periodic orbits. Finally, the proposed general theoretical results are successfully applied to a ``food-limited" population model and a weak Allee effect-driven population model, each of which incorporates diffusion, time delay, and advection, thus confirming the validity of our approach.}

\end{abstract}

\newpage
\section {Introduction}

In recent decades, the interaction between diffusion and time delays has drawn considerable attention in the fields of mathematical biology and dynamical systems. Many existing works (see, e.g., \cite{hu2011spatially, lou2006effects, Yan2010, Yan2012, su2009hopf, su2012hopf}) have been devoted to the dynamical behavior near spatially homogeneous or nonhomogeneous steady-state solutions. It should be pointed out, however, that results concerning spatially nonhomogeneous steady-state solutions remain relatively scarce, largely due to the considerable challenges involved in their analysis. There are two main reasons for this difficulty: first, establishing the existence of nontrivial steady-state solutions is itself a nontrivial task, often requiring sophisticated tools such as a priori estimates and topological degree theory; second, the characteristic equation associated with the linearized system generally lacks an explicit algebraic structure when the steady state is spatially nonuniform, which greatly complicates the eigenvalue analysis. Therefore, a thorough investigation into the stability and bifurcation of spatially nonhomogeneous steady-state solutions is not only of theoretical importance but also of practical relevance in applications. 

In spatially heterogeneous environments, organisms typically possess the ability to sense and respond to local conditions through adaptive movement, which often involves both random diffusion and directional advection.(see, e.g., \cite{averill2017role, lam2011concentration, lam2012limiting, belgacem1995effects, cantrell2003spatial, cosner2014reaction, lam2015evolution, lou2014evolution, lou2016qualitative, lou2019global, lou2015evolution, vasilyeva2012how, zhou2018evolution}) To accurately capture such taxis behavior, it is essential to incorporate advection effects into population models. In this context, we consider the following single-species model with a general advection term:
\begin{equation}\label{1.1}
\begin{cases}
u_t = D\nabla \cdot [d({x})\nabla u(x,t) - \vec{b}({x})u(x,t)] + u(x,t)f\left(x, u(x,t-r)\right), & {x} \in \Omega, \; t > 0, \\
u = 0, & {x} \in \partial\Omega, \; t > 0.
\end{cases}
\end{equation}
where $u(x,t)$ represent the density of the population at position $x$ and time $t$, and $r\geq0$ is the time-delay parameter, often used to represent the maturity time of the species, and D represents the average rate of diffusion and advection, and $\Omega$ is a bounded domain in $\mathbb{R}^n$ $(1 \le n \le 3)$ with a smooth boundary $\partial\Omega$. 

It is worth noting that if there exists a function $B(x)$ such that $\vec{b}(x) = d(x) \nabla B(x)$, then the problem can be significantly simplified through variable substitution. However, this condition is not assumed to be true in this paper, so the adopted approach has a wider range of applicability. Since the elliptic operator considered in this study is a non-self-adjoint operator and cannot be transformed into a self-adjoint form through transformation, the powerful tools applicable to self-adjoint operators in classical eigenvalue theory cannot be directly applied here, which brings certain technical challenges to the analysis.

Our proposed model is capable of capturing not only the classical logistic growth but also also be used to consider growth models such as “food-limited” and weak Allee effect. The logistic growth, a widely used model, describes population dynamics under crowding effects driven by intraspecific competition. The food-limited growth model constitutes a refinement of the classical logistic growth. It describes a growth limitation mechanism based on the proportion of unexploited resources, thereby providing a more realistic representation of population dynamics under resource scarcity. (see, e.g.,\cite{gopalsamy1988time, gopalsamy1990environmental, kuang1993delay} ). Meanwhile, the weak Allee effect describes the phenomenon where the reproductive and survival rates of individuals decrease in populations that are relatively small (see, e.g., \cite{allee1931animal, cantrell2003spatial, shi2006persistence, stephens1999what}).

Letting $t = \tilde{t}/D$, dropping the tilde sign, and denoting $\lambda = 1/D$ and $\tau = Dr$, we see that \eqref{1.1} can be transformed to the following equivalent model:
\begin{equation}\label{delay}
\begin{cases}
u_t =\nabla \cdot[d(x) \nabla u(x,t)-\vec{b}(x) u(x,t)]+
\la u(x,t)f\left(x, u(x,t-\tau)\right),&x\in\Om,\ t>0,\\
u(x,t)=0,&  x\in\partial\Om,\ t>0.\\
\end{cases}
\end{equation}
Throughout the paper, unless otherwise specified, $f\left(x, u\right)$ satisfies the following assumption:
\begin{enumerate}
\item[$\mathbf{(H_1)}$] $f\left(x, u\right)$ is a smooth function (for example, $f$ has 
fourth order continuous partial derivatives) ,$f\left(x, 0\right):=m(x)$, $m(x)\in C^2(\overline\Omega)$, and $\max_{x\in\overline\Omega}m(x)>0$.
\end{enumerate}
\begin{enumerate}
\item[$\mathbf{(H_2)}$]
There exists  $\alpha \in \left(0, \frac{1}{2}\right)$  such that 
$m(x) \in C^{\alpha}(\bar{\Omega})$ and 
$\vec{b}(x) \in \left[C^{1+\alpha}(\bar{\Omega})\right]^n $, 
 where  $n \leq 3$  is the dimension of the domain $ \Omega$. 
 Moreover, $ d(x) > 0 \text{ for } x \in \bar{\Omega}$.
\end{enumerate}

Busenberg and Huang \cite{busenberg1996stability} are pioneers in studying the dynamics near nonhomog-\\-eneous steady-state solutions in space. They utilized the perturbation method and the implicit function theorem to study a class of time-delay diffusion Hutchinson equations with Dirichlet boundary value conditions,and established the theoretical framework for the existence of non-homogeneous steady-state solutions and Hopf bifurcations:
\begin{equation*} 
\left\{
\begin{array}{ll}
\displaystyle
u_t = d\Delta u(x,t) + \lambda u(x,t)[1 - u(x,t-\tau)], & x \in \Omega,\ t > 0, \\[2ex]
u(x,t) = 0, & x \in \partial\Omega,\ t > 0.
\end{array}
\right.
\end{equation*}

 In the one-dimensional spatial region $\Omega = (0,\pi)$, they proved that when the parameter $\lambda > d$ and approaches $d$, the unique spatial non-homogeneous positive equilibrium solution will lose stability under large time delays $\tau$ and exhibit Hopf bifurcation, thereby triggering the oscillation behavior of the system. 

Inspired by the research of Busenberg and Huang, Su et al. \cite{su2009hopf} extended the research scope to a class of time-delay reaction-diffusion population models with a more general form of per capita growth rate, and systematically analyzed their Hopf bifurcation behavior. The relevant results of this model in the high-dimensional case can be found in the work of Yan and Li \cite{yan2005}. Subsequently, Chen and Shi \cite{Chen2012} further studied the stability and Hopf bifurcation issues of positive non-homogeneous steady-state solutions in the Logistic-type population model with non-local diffusion terms. Chen and Yu \cite{ChenYu2016} generalized the conclusions in \cite{busenberg1996stability} to non-local time-delay reaction-diffusion population models. Based on this, Jin and Yuan \cite{jin2021}  established the Hopf bifurcation theory for a more general type of reaction-diffusion-migration equations. It is worth noting that Guo and Ma \cite{guo2016} innovatively applied the Lyapunov–Schmidt reduction method to study the dynamical behavior of delay diffusion equations under Dirichlet boundary conditions. One of the significant advantages of this method is that the obtained non-homogeneous steady-state solutions have explicit algebraic expressions, thereby providing important convenience for subsequent stability and bifurcation analysis.

Recently, Liu and Chen \cite{liu2022} considered a delayed rereaction diffusion model with a general advection term:
\begin{equation*}
\begin{cases}
\displaystyle
u_t = \nabla \cdot [d(x) \nabla u - \vec{b}(x) u] + \lambda u[m(x) - a(x) u(t-\tau)], & x \in \Omega, \; t > 0, \\
u = 0, & x \in \partial\Omega, \; t > 0.
\end{cases}
\end{equation*}

it provides important inspiration for our current research. On this basis, we consider the Hopf bifurcation problem of the reaction-diffusion-advection equation with the general advection term and the general per capita growth rate in Model (1.2). This generalization has a broader applicable context. Meanwhile, due to the computational complexity brought about by general nonlinear terms, this generalization is theoretically non-trivial.

The rest of the paper is organized as follows. In Section 2, we provide some preliminary knowledge to facilitate our subsequent proofs.In Section 3.1, we obtained the existence of spatially nonhomogeneous steady-state solution of equation \eqref{delay} by using the Lyapunov-Schmidt reduction. In Section 3.2, we discussed the eigenvalue problem of the infinitesimal generator of the $C_0$-semigroup generated by the solution of the linearized system and presented the corresponding results. In Section 3.3, by using the local Hopf bifurcation theory, we obtained some results on the stability of spatially nonhomogeneous steady-state solutions and the existence of Hopf bifurcation phenomena. In Section 4, we derived an explicit formula which can be used to determine the stability and bifurcation direction of the Hopf bifurcation periodic solutions near the non-homogeneous steady-state solutions.In Section 5.1, we separately considered the weak Allee effect and the "food-limited" population model, and presented some numerical results inspired by our theoretical research. Finally, in Section 5.2, we summarized our theoretical findings.

\section{Some preliminaries}
In this section, we provide a preliminary introduction to the concept of space, its corresponding inner product, and a class of eigenvalue problems. Throughout the paper, we also denote the spaces 
\begin{equation*}
    \mathbb{X}=H^2(\Om)\cap H^1_0(\Om),\quad H_{0}^{1}(\Omega) = \left\{ u \in H^{1}(\Omega) \mid u(x) = 0 \text{ for all } x \in \partial\Omega \right\},\quad \mathbb{Y}=L^2(\Om),
\end{equation*}
and
\begin{equation*}
     C=C([-\tau,0],\mathbb{Y}),\mathcal{C}=C([-1,0],\mathbb{Y}).
\end{equation*}

Moreover, we denote the complexification of a linear space $Z$ to be 
\begin{equation*}
    Z_\mathbb{C}:= Z\oplus iZ=\{x_1+ix_2|~x_1,x_2\in Z\}
\end{equation*}
 the domain of a linear operator $L$ by $\mathscr{D}(L)$, the kernel of $L$ by $\mathscr{N}(L)$, and the range of $L$ by $\mathscr{R}(L)$.
For Hilbert space $Y_{\mathbb C}$, we use
the standard inner product $\langle u,v \rangle=\ds\int_{\Om} \overline
u(x) {v}(x) dx$.
Moreover, for space $X_{\mathbb{C}}$, we use the inner product:
\begin{equation*}
    \langle u, v\rangle_{X} = \sum_{|\alpha|\leq 2} \int_{\Omega} \overline{(D^{\alpha}u)} D^{\alpha}v  dx \quad \text{for } u, v \in X_{\mathbb{C}}
\end{equation*}
where $|\alpha| = \sum_{i=1}^{n} \alpha_{i}$ and $D^{\alpha}u = \frac{\partial^{|\alpha|}u}{\partial x_{1}^{\alpha_{1}} \cdots \partial x_{n}^{\alpha_{n}}}$,
and the norm is $\|u\|_{X_{\mathbb{C}}} = \sqrt{\langle u,u\rangle_{X}}$ for $u \in X_{\mathbb{C}}$.

For space $Y_{\mathbb{C}}$, we use the inner product:
\[
\langle u, v\rangle_{Y_{\mathbb{C}}} = \int_{\Omega} \overline{u}(x) v(x)  dx \quad \text{for } u, v \in Y_{\mathbb{C}},
\]
and the norm is $\|u\|_{Y_{\mathbb{C}}} = \sqrt{\langle u,u\rangle_{Y_{\mathbb{C}}}}$ for $u \in Y_{\mathbb{C}}$.

We cite the following two results from \cite{liu2022} ,\cite{{Cantrell2003}} ,\cite{{Hess1991}},which are crucial for our subsequent proofs.
\begin{lemma}
Assume that $\mathbf{(H_1)}$, $\mathbf{(H_2)}$ holds. Then the eigenvalue problem
\begin{equation}\label{2.5} 
\begin{cases}
    -\nabla \cdot[d(x) \nabla \psi-\vec{b}(x) \psi]  = \la m(x)\psi, & x \in \Omega, \\
    \psi = 0, & x \in \partial\Omega.
\end{cases}     
\end{equation}
admits a unique positive principal eigenvalue $\lambda^{*}$,which is also the principal eigenvalue for
\begin{equation}\label{2.6}
\begin{cases}
    -\nabla \cdot[d(x) \nabla \psi]-\vec{b}(x)\cdot\nabla\psi  = \la m(x)\psi, & x \in \Omega, \\
    \psi = 0, & x \in \partial\Omega.
\end{cases}     
\end{equation}
Let $\phi$ and $\phi^*$ represent the principal eigenfunctions of \eqref{2.5} and \eqref{2.6}, and they satisfy:
\begin{equation*}
    \phi, \phi^{*} > 0 \text{ for } x \in \Omega, \text{ and } \int\limits_{\Omega} \phi  dx = \int\limits_{\Omega} \phi^{*}  dx = 1.
\end{equation*}
\end{lemma}
\begin{lemma}\label{lemma2.2}
Assume that $\mathbf{(H_1)}$, $\mathbf{(H_2)}$ holds, and let  $\lambda_{*}$ and $\phi$ and $\phi^{*}$ be defined in Lemma 2.1, then
\begin{equation*}
\frac{\int_{\Omega}m(x)\phi\phi^{*}dx}{\int_{\Omega}\phi\phi^{*}dx}>0,
\end{equation*}  
\end{lemma}

\section {Steady-State Bifurcation}
\subsection{ Existence of Steady-State Solution}

\quad In this section, we first consider the existence of positive steady states of Eq.
\eqref{delay}, which satisfy:
\begin{equation} \label{steady}
\begin{cases}
\nabla \cdot[d(x) \nabla u-\vec{b}(x) u]+
\la uf\left(x, u\right)=0,&x\in\Om,\ t>0,\\
u(x,t)=0,&  x\in\partial\Om,\ t>0.\\
\end{cases}
\end{equation}

Define $F:\mathbb{X}\times\mathbb{R}\to\mathbb{Y}$ by

\begin{equation*}
F(u,\lambda)=\nabla \cdot[d(x) \nabla u-\vec{b}(x) u]+
\la uf\left(x, u\right)
\end{equation*}
\begin{equation}\label{LL}
L_{\lambda}\phi:=\nabla \cdot[d(x) \nabla \phi-\vec{b}(x) \phi] +\la m(x)\phi,
\end{equation}
and the adjoint operator $L^*_{\lambda}$ of $L_{\lambda}$ is as follows:
\begin{equation*}\label{L*}
L^*_{\lambda}\phi:=\nabla \cdot[d(x) \nabla \phi]+\vec{b}(x)\cdot\nabla\phi +\la m(x)\phi,
\end{equation*}

Note that
\begin{equation*}
      \mathbb{X}  =  \mathscr{N}\left(L_{\lambda_{*}}\right)\oplus X_1,\;\;
      \mathbb{Y}  = \mathscr{N}\left(L^{*}_{\lambda_{*}}\right)\oplus Y_1,
\end{equation*}
where
\begin{equation*}\label{L}
\begin{split}
\mathscr{N}\left(L_{\lambda_{*}}\right)=&\text{span}\{\phi\},\;\;
X_1=\left\{y\in X:\int_{\Om}\phi(x) y(x)dx=0\right\},\\
Y_1=&\mathscr{R}\left(L_{\lambda_{*}}\right)=\left\{y\in Y:\int_{\Om}\phi^{*}(x)
y(x)dx=0\right\}.
\end{split}
\end{equation*}

Now, we use Lyapunov-Schmidt reduction methods as follows. Let $Q$ and $I-Q$ denote the projection operators from $\mathbb{Y}$ onto $Y_1$ and $\mathscr{N}\left(L^{*}_{\lambda_{*}}\right)$, respectively. Thus, $F(u,\lambda)=0$ is equivalent to the following system:
\begin{equation}\label{project}
\left\{
    \begin{aligned}
        & Q F(u_1 + u_2, \lambda) = 0, \\
        & (I - Q) F(u_1 + u_2, \lambda) = 0.
    \end{aligned}
\right.
\end{equation}

where $u_{1}\in\mathscr{N}\left(L_{\lambda_{*}}\right)$ and $u_{2}\in X_1$. Notice that $F(0,\lambda_{*})=0$ and $QF_{u_{2}}(0,\lambda_{*})=L_{\lambda_{*}}$. Applying the implicit function theorem, we obtain a continuously differentiable map $h:\mathfrak{U}\rightarrow {X}_{1}$ such that

\begin{equation*}\label{2.61}
h(0,\lambda)=0\quad,h_u^{\prime}(0,\lambda_*)=0\quad\text{and}\quad QF(u_{1}+h(u_{1},\lambda),\lambda)\equiv 0,
\end{equation*}

where $\mathfrak{U}$ is an open neighborhood of $(0,\lambda_{*})$ in $\mathscr{N}\left(L_{\lambda_{*}}\right)\times\mathbb{R}$. Substituting $u_{2}=h(u_{1},\lambda)$ into the second equation of \eqref{project} gives

\begin{equation*}
\mathfrak{F}(u_{1},\lambda):=(I-Q)F(u_{1}+h(u_{1},\lambda),\lambda)=0.
\end{equation*}

Thus, each solution to $\mathfrak{F}(u_{1},\lambda)=0$ in $\mathfrak{U}$ one-to-one corresponds to some solution to $F(u,\lambda)=0$.

For $u_{1}=t\phi\in\mathscr{N}\left(L_{\lambda_{*}}\right)$ with $t\in\mathbb{R}$, substituting this into (2.5) and then calculating the inner product with $\phi^{*}$ on $\Omega$, we have $\beta( t,\lambda)=0$, where $\beta:\mathbb{R}^{2}\rightarrow\mathbb{R}$ is explicitly given by:
\begin{equation*}
\begin{aligned}
&\beta( t,\lambda) := \langle\phi^{*}, F(t\phi(x)+h(t\phi(x),\lambda),\lambda) \rangle 
\\&=\int_{\Omega}\phi^{*}(x)F(t\phi(x)+h(t\phi(x),\lambda),\lambda)\mathrm{d}x.    
\end{aligned}
\end{equation*}

Noting that $\beta( 0,\lambda) = 0$,$\beta^{\prime}( 0,\lambda^*) = 0$ we take the following function
\begin{equation*}
H( t,\lambda) =
\begin{cases}
\frac{\beta( t,\lambda)}{t} & \text{if } t \neq 0 \\
\beta_{t}^{\prime}(\lambda, 0) & \text{if } t = 0
\end{cases}
\end{equation*}

After a simple calculation, it can be obtained that:
\begin{equation*}
 H_{\lambda}^{\prime}( 0,\lambda^{*}) = \beta_{t,\lambda}^{\prime\prime}( 0,\lambda^{*}) = \langle \phi^{*}, F_{u,\lambda}^{\prime\prime}( 0,\lambda^{*})[\phi] \rangle=
\int_{\Omega} m(x)\phi\phi^{*}dx\stackrel{\text{def}}{=}a
\end{equation*}

From Lemma 2.2, we know that $H_{\lambda}^{\prime}(0,\lambda^{*})\neq 0$. Applying the implicit function theorem to $h$ we find $\lambda=\lambda(t)$, defined in an $\varepsilon$-neighbourhood of $t=0$, such that
\begin{equation*}
\lambda(0)=\lambda^{*},\quad H(t,\lambda(t))=0,\quad\forall -\varepsilon\leq t\leq\varepsilon.    
\end{equation*}
\begin{equation*}
\lambda^{\prime}(0)=-\frac{H_{t}^{\prime}(0,\lambda^{*})}{H_{\lambda}^{\prime}(0,\lambda^{*})}.
\end{equation*}

In addition, one easily finds that
\begin{equation*}
H_{t}^{\prime}(0,\lambda^{*})=\frac{1}{2}\beta_{t,t}^{\prime\prime}(0,\lambda^{*})=\frac{1}{2}\langle\phi^{*}, F_{u,u}^{\prime\prime}(0,\lambda^{*})[\phi]^{2}\rangle=
\lambda_{*} \int_{\Omega} f^{\prime}_u(x, 0) \phi^{2} \phi^{*}  \mathrm{d} x\stackrel{\text{def}}{=}b.    
\end{equation*}

Hence, if $b\neq 0$ we get
\begin{equation*}
\lambda(t)=\lambda^{*}-\frac{b}{a}t+o(t),\quad\text{as } t\to 0.    
\end{equation*}

 If $b = 0$ , one gets
\begin{equation*}
\begin{aligned}
\lambda^{\prime\prime}(0) &= -\frac{1}{3a} (\langle \phi^{*}, F_{u,u,u}^{\prime\prime\prime}( 0,\lambda^{*})[\phi]^{3} \rangle+3\langle \phi^{*}, F_{u,u}^{\prime\prime}( 0,\lambda^{*}) h_{u_{1}u_{1}}^{\prime\prime}(0,\lambda_*)[\phi]^{3} \rangle)
\\&=-\frac{1}{3a}(3\lambda_{*} \int_{\Omega} f^{\prime\prime}_{uu}(x, 0) \phi^{3} \phi^{*}  \mathrm{d} x+6\lambda_{*} \int_{\Omega} f^{\prime}_u(x, 0) h_{u_{1}u_{1}}^{\prime\prime}(0,\lambda_*)\phi^{3} \phi^{*}  \mathrm{d} x).    
\end{aligned}
\end{equation*}

We still need to compute $h_{u_{1}u_{1}}^{\prime\prime}(0,\lambda_*)$. It follows from Eq. \eqref{LL} that  
\begin{equation*}
    {L}_{\lambda_{*}} h_{u_{1}u_{1}}^{\prime\prime}(0,\lambda_*)[\phi]^{2} 
     + Q F_{uu}^{\prime\prime}(\phi, \lambda_{*}) = 0
\end{equation*}
 and hence that $h_{u_{1}u_{1}}^{\prime\prime}(0,\lambda_*)[\phi]^{2} = -2\la_* f^{\prime}_u(x, 0) {L}_{\lambda_{*}}^{-1} [\phi^{2}]$.

Thus, if $b = 0$,$\lambda^{\prime\prime}(0) \neq 0$ one finds
\begin{equation*}
u = \pm \left( \frac{2(\lambda - \lambda^{*})}{\lambda^{\prime\prime}(0)} \right)^{1/2} \phi + O(\lambda - \lambda^{*})   
\end{equation*}

Thus, we have the following result.

\begin{theorem}\label{2.1}
If $\ds\int_{\Om}f^{\prime}_{u}(x, 0)\phi^2\phi^{*}dx \neq 0$, then there exist a constant $\epsilon > 0$ and a continuously differentiable mapping $\lambda \to t_\lambda$ from $(\lambda_* - \epsilon, \lambda_* + \epsilon)$ to $\mathbb{R}$ such that Eq. \eqref{steady} has a nontrivial solution
\begin{equation*}
    u_\lambda(x) = t_\lambda \phi(x) + h(t_\lambda \phi(x), \lambda)
\end{equation*}
where $t_\lambda=(\la-\la_*)\beta_{\la}+ o(\lambda - \lambda^{*})$ satisfies $\lim_{\lambda \to \lambda_*} t_\lambda = 0$, which exists for $\lambda \in (\lambda_{*} - \epsilon, \lambda_{*}) \cup (\lambda_{*}, \lambda_{*} + \epsilon)$ and satisfies
\begin{equation*}
\lim_{\lambda \to \lambda_{*}} u_{\lambda} = 0
\end{equation*}

Moreover, for $\la=\la_*$,
\begin{equation}\label{al}
\beta_{\la_*}
=\ds\f{\ds\int_{\Om}m(x)\phi\phi^{*}dx}{-\la_*\ds\int_{\Om}f^{\prime}_{u}(x, 0)\phi^2\phi^{*}dx},
\end{equation}

\end{theorem}
\begin{remark}
Since $\phi_{1}>0$ on $\Omega$, the sign of the spatially nonhomogeneous steady-state solution $u_{\lambda}$ in Theorem \ref{2.1} coincides with the sign of $t_\lambda$: $u_{\lambda}$ is positive when $t_\lambda >0$  and negative when $t_\lambda <0$ . For the representative case $f(x,u)=m(x)-u$ (where $f^{\prime}_{u}(x, 0)=-1$), the solution is positive for $\lambda$ $\in$ ($\lambda_{*}, \lambda_{*} + \epsilon$). From a biological standpoint, only positive steady-state solutions are of interest.
\end{remark}

\begin{theorem}\label{2.3}
If $\ds\int_{\Om}f^{\prime}_{u}(x, 0)\phi^2\phi^{*}dx =0$ and $\lambda^{\prime\prime}(0)>0$ (respectively, $\lambda^{\prime\prime}(0)<0$), then there exist a constant $\lambda^{*}>\la_*$ (respectively, $\lambda^{*}<\la_*$) and two continuously differentiable mappings $\lambda\to t_{\lambda}^{\pm}$ from $(\la_*,\lambda^{*}]$ to $\mathbb{R}$ (respectively, from $[\lambda^{*},\la_*]$ to $\mathbb{R}$) such that Eq. (2.1) has two nontrivial solutions
\begin{equation*}
u_{\lambda}^{\pm}(x)=t_{\lambda}^{\pm}\phi(x)+h\left(t_{\lambda}^{\pm}\phi(x),\lambda\right),    
\end{equation*}
and satisfies $\lim_{\lambda\to\la_*}u_{\lambda}^{\pm}=0.$
\end{theorem}

\begin{remark}
We see that one of the two spatially nonhomogeneous steady-state solutions $u_{\lambda}^{\pm}$ established by throrem \ref{2.3} is positive and the other is negative .
\end{remark}
\subsection{Eigenvalue Problem}

 For $\epsilon > 0$, define $\Lambda_{1\epsilon} = (\lambda_*-\epsilon,\lambda_*)$ and $\Lambda_{2\epsilon} = (\lambda_*,\lambda_*+\epsilon)$. By Theorems \ref{2.1} and \ref{2.3}, there exists $\epsilon > 0$ such that \eqref{steady} has a spatially nonhomogeneous steady-state solution $u_{\lambda}$ for $\lambda \in \Lambda$, where $\Lambda = \Lambda_{1\epsilon} \cup \Lambda_{2\epsilon}$  if  $\int_{\Omega}f^{\prime}_{u}(x, 0)\phi^2\phi^{*}dx \neq 0$, and if $\int_{\Omega}f^{\prime}_{u}(x, 0)\phi^2\phi^{*}dx = 0$, then $\Lambda = \Lambda_{2\epsilon}(\Lambda_{1\epsilon})$  when $\lambda^{\prime\prime}(0)>0(<0)$.

In subsequent sections, $u_{\lambda}(x)$ denotes this steady-state solution for $\lambda \in \Lambda$.
\begin{equation*}
u_{\lambda}(x)=t_{\lambda}\phi(x)+h(t_{\lambda}\phi(x),\lambda),
\end{equation*}

Linearizing system \eqref{delay} at $u_\la$, we have
\begin{equation*}
\label{linear}\begin{cases}
\    \begin{aligned}
        v_t
        &= \nabla \cdot \Big[d(x) \nabla v - \vec{b}(x) v\Big] + \lambda f(x, u_{\lambda}) v(x, t) \\
        &\quad + \lambda u_{\lambda} f'(x, u_{\lambda}) v(x, t-\tau),
    \end{aligned} 
    & x \in \Omega,\ t > 0, \\[3ex]
    v(x, t) = 0, 
    & x \in \partial \Omega,\ t > 0.
\end{cases}
\end{equation*}

It follows from \cite{wu1996theory} that the solution semigroup of
Eq. \eqref{linear} has the infinitesimal generator $A_\tau(\la)$
satisfying
\begin{equation*}\label{Ataula}
A_\tau(\la) \Psi=\dot\Psi,
\end{equation*}
where
\begin{equation*}
\begin{aligned}
\mathscr{D}\left(A_{\tau}(\lambda)\right) = \left\{ \Psi \in C_{\mathbb{C}} \cap C_{\mathbb{C}}^{1} : \Psi(0) \in X_{\mathbb{C}},\ \dot{\Psi}(0) = \nabla \cdot [d(x) \nabla \Psi(0) - \vec{b}(x) \Psi(0)] \right. \\
\left. + \lambda f\left(x, u_{\lambda}\right) \Psi(0) + \lambda f^{\prime}_u\left(x, u_{\lambda}\right) \Psi(-\tau) \right\}
\end{aligned}
\end{equation*}
and $
C^1_\mathbb{C}=C^1([-\tau,0],\mathbb{Y}_\mathbb{C})$. Moreover, $\mu\in\mathbb{C}$ is an eigenvalue of $A_\tau(\la)$, if and only if there exists $\psi(\ne0)\in
X_{\mathbb{C}}$ such that $\Delta(\la,\mu,\tau)\psi=0$,
where
\begin{equation}\label{triangle}
\begin{split}
&\Delta(\la,\mu,\tau)\psi:\\
=&\nabla \cdot[d(x) \nabla \psi-\vec{b}(x) \psi] +\la f\left(x,u_{\lambda}\right)\psi+ \lambda u_{\lambda} f^{\prime}_u\left(x,u_{\lambda}\right)\psi e^{-\mu\tau}-\mu\psi.
\end{split}
\end{equation}

We will establish that as the time delay $\tau$ increases, eigenvalues of $A_\tau(\la)$ can cross the imaginary axis. Specifically, $A_\tau(\la)$ possesses a purely imaginary eigenvalue $\mu = i\nu$  (with  $\nu$ $>$ $0$ ) for some $\tau\ge0$  if and only if :
\begin{equation}\label{eigen}
\nabla \cdot[d(x) \nabla \psi-\vec{b}(x) \psi] +\la f\left(x,u_{\lambda}\right)\psi+ \lambda u_{\lambda} f^{\prime}_u\left(x,u_{\lambda}\right)\psi  e^{-i\theta}-i\nu\psi=0
\end{equation}
is solvable for some value of $\nu>0$, $\theta\in[0,2\pi)$, and $\psi(\ne 0)\in X_{\mathbb{C}}$.

We first establish the following estimates for solutions to \eqref{eigen}. The proof of this lemma was inspired by \cite{liu2022}.

\begin{lemma}\label{nu}
If $(\psi_\lambda,\nu_\lambda,\theta_\lambda)$ solves \eqref{eigen}, where $\psi_\lambda(\neq 0)\in X_{\mathbb{C}}$, $\nu_\lambda>0$, and $\theta_\lambda\in[0,2\pi)$, then $\frac{\nu_\lambda}{|\lambda-\lambda_{*}|}$ is bounded for $\lambda\in\Lambda$. Moreover, ignoring a scalar factor, $\psi_\lambda$ can be represented as
\begin{equation}\label{2.20}
\begin{cases}
\psi_\lambda=r_\lambda\phi+w_\lambda, & r_\lambda\geq 0,\ w_\lambda\in(X_{1})_{\mathbb{C}}, \\
\|\psi_\lambda\|_{\mathbb{X}}=\|\phi\|_{\mathbb{X}},
\end{cases}
\end{equation}
where $r_\lambda$, $w_\lambda$ and $\psi_\lambda$ satisfy $\lim_{\lambda\to\lambda_{*}}r_\lambda=1$, $\lim_{\lambda\to\lambda_{*}}w_\lambda=0$, $\lim_{\lambda\to\lambda_{*}}\psi_\lambda=\phi$ in $(C^{2,\alpha}(\overline{\Omega}))_{\mathbb{C}}$, and $\alpha$ and $\phi$ are defined in $(\mathbf{H}_{1})$ and (2.5), respectively.
\end{lemma}

\begin{proof}
Ignoring a scalar factor, $\psi_\lambda$ can be represented as \eqref{2.20}. The proof proceeds in three steps.

\noindent\textbf{Step 1.} We prove that $\nu_\lambda$ is bounded for $\lambda \in \Lambda$. Suppose, for contradiction, that it is unbounded. Then there exists a sequence $\{\lambda_{n}\}_{n=1}^{\infty}$ such that
\[
\lim_{n\to\infty}\lambda_{n} = \lambda_{*}, \quad \lim_{n\to\infty}\nu_{\lambda_{n}} = +\infty.
\]

Since $\theta_\lambda$ is bounded for $\lambda \in \Lambda$ and $X$ is compactly embedded into $C^{\alpha}(\overline{\Omega})$ (with $\alpha \in (0,1/2)$ from (H2)), there exists a subsequence $\{\lambda_{n_{k}}\}_{k=1}^{\infty}$ (denoted as $\{\lambda_{n}\}_{n=1}^{\infty}$ for simplicity) such that
\begin{equation} \label{2.21}
\lim_{n\to\infty}\theta_{\lambda_{n}} = \theta_{*} \in [0,2\pi], \quad \text{and} \quad \lim_{n\to\infty}\psi_{\lambda_{n}} = \psi_{*} \text{ in } (C^{\alpha}(\overline{\Omega}))_{\mathbb{C}}.
\end{equation}

Substituting $(\psi, \nu, \theta, \lambda) = (\psi_{\lambda_{n}}, \nu_{\lambda_{n}}, \theta_{\lambda_{n}}, \lambda_{n})$ into \eqref{eigen}, multiplying by $\overline{\psi_{\lambda_{n}}}$, and integrating over $\Omega$ yields:
\begin{align}
&\langle \psi_{\lambda_{n}}, L_{\lambda_{*}} \psi_{\lambda_{n}} \rangle - \int_{\Omega} \left[ \lambda_{*} m(x) |\psi_{\lambda_{n}}|^{2} - \lambda_{n} f(x, u_{\lambda_{n}}) u_{\lambda_{n}} |\psi_{\lambda_{n}}|^{2} \right] dx \nonumber \\
&\quad + \lambda_{n} e^{-i\theta_{\lambda_{n}}} \int_{\Omega} f^{\prime}(x, u_{\lambda_{n}}) u_{\lambda_{n}} |\psi_{\lambda_{n}}|^{2} dx - i \nu_{\lambda_{n}} \int_{\Omega} |\psi_{\lambda_{n}}|^{2} dx = 0,
\end{align}
where $L_{\lambda_{*}}$ is defined in (2.3). Using the expression
\[
\langle \psi_{\lambda_{n}}, L_{\lambda_{*}} \psi_{\lambda_{n}} \rangle = -\int_{\Omega} d(x) |\nabla \psi_{\lambda_{n}}|^{2} dx + \int_{\Omega} \psi_{\lambda_{n}} \vec{b}(x) \cdot \nabla \overline{\psi_{\lambda_{n}}} dx + \lambda_{*} \int_{\Omega} m(x) |\psi_{\lambda_{n}}|^{2} dx,
\]

we derive the estimates:
\begin{align} \label{2.23}
\nu_{\lambda_n} \| \psi_{\lambda_n} \|_2 &\leq \max_{x \in \overline{\Omega}} |\vec{b}(x)| \| |\nabla \psi_{\lambda_n}| \|_2 + \lambda_n \max_{x \in \overline{\Omega}} |f^{\prime}(x, u_{\lambda_{n}})| \| u_{\lambda_n} \|_{\infty} \| \psi_{\lambda_n} \|_2, \\
\int_{\Omega} d(x) |\nabla \psi_{\lambda_n}|^2 dx &\leq \max_{x \in \overline{\Omega}} |\vec{b}(x)| \| |\nabla \psi_{\lambda_n}| \|_2 \| \psi_{\lambda_n} \|_2 + \lambda_n \int_{\Omega} |f(x, u_{\lambda_{n}})| |\psi_{\lambda_n}|^2 dx \nonumber \\
&\quad + \lambda_n \max_{x \in \overline{\Omega}} |f^{\prime}(x, u_{\lambda_{n}})| \| u_{\lambda_n} \|_{\infty} \| \psi_{\lambda_n} \|_2^2.
\end{align}

Since $\| \psi_{\lambda_n} \|_{X}$$ =$ $\| \phi \|_{X}$, the sequences $\{ \| |\nabla \psi_{\lambda_n}| \|_{2} \}$$_{n=1}^{\infty}$ and $\{ \| \psi_{\lambda_n}$$ \|_{2} \}$$_{n=1}^{\infty}$ are bounded. As $\lim_{n\to\infty}$$ \nu_{\lambda_n}$$ = +\infty$, the first inequality in \eqref{2.23} implies $\lim_{n\to\infty}$$ \| \psi_{\lambda_n}$$ \|_{2} = 0$. Combined with \eqref{2.21}, this gives $\psi_{*}$$ = 0$ and $\lim_{n\to\infty}$$ \psi$$_{\lambda_n} = 0$ in $(C^{\alpha}(\overline{\Omega}))_{\mathbb{C}}$. From the second inequality in \eqref{2.23}, we obtain:
\[
\lim_{n\to\infty} \int_{\Omega} |\nabla \psi_{\lambda_n}|^2 dx = 0.
\]

Thus, $\lim_{n\to\infty} \nu_{\lambda_n} \psi_{\lambda_n} = 0$ in $Y_{\mathbb{C}}$. Rewriting \eqref{eigen} as:
\begin{equation*}
(L_{\lambda_{*}} - I) \psi_{\lambda_n} + M_{\lambda_n} \psi_{\lambda_n} = 0,
\end{equation*}
where
\[
M_{\lambda_n} \psi_{\lambda_n} = \psi_{\lambda_n} - \lambda_* m(x) \psi_{\lambda_n} + \lambda_n f(x, u_{\lambda_n}) \psi_{\lambda_n} + \lambda_n f^{\prime}(x, u_{\lambda_n}) u_{\lambda_n} \psi_{\lambda_n} e^{-i\theta_{\lambda_n}} - i \nu_{\lambda_n} \psi_{\lambda_n},
\]
and noting $\lim_{n\to\infty} M_{\lambda_n} \psi_{\lambda_n} = 0$ in $Y_{\mathbb{C}}$, the continuity of $(L_{\lambda_{*}} - I)^{-1}$ from $Y_{\mathbb{C}}$ to $X_{\mathbb{C}}$ implies:
\[
\lim_{n\to\infty} \psi_{\lambda_n} = -\lim_{n\to\infty} (L_{\lambda_{*}} - I)^{-1} M_{\lambda_n} \psi_{\lambda_n} = 0 \text{ in } X_{\mathbb{C}}.
\]

This contradicts $\| \psi_{\lambda_n} \|_{X} = \| \phi \|_{X}$, proving boundedness of $\nu_\lambda$.

\noindent\textbf{Step 2.} We show $\lim_{\lambda \to \lambda_*} \nu_\lambda = 0$ and $\lim_{\lambda \to \lambda_*} \psi_\lambda = \phi$ in $(C^{2,\alpha}(\overline{\Omega}))_{\mathbb{C}}$. Consequently,
\[
\lim_{\lambda \to \lambda_*} r_\lambda = 1, \quad \text{and} \quad \lim_{\lambda \to \lambda_*} w_\lambda = 0 \text{ in } (C^{2,\alpha}(\overline{\Omega}))_{\mathbb{C}}.
\]

For any sequence $\{\lambda_k\}_{k=1}^{\infty}$ with $\lim_{k \to \infty} \lambda_k = \lambda_*$, there exists a subsequence $\{\lambda_{k_n}\}_{n=1}^{\infty}$ (denoted as $\{\lambda_n\}_{n=1}^{\infty}$) such that:
\[
\lim_{n \to \infty} \nu_{\lambda_n} = a \geq 0, \quad \lim_{n \to \infty} \theta_{\lambda_n} = \widetilde{\theta} \in [0,2\pi], \quad \text{and} \quad \lim_{n \to \infty} \psi_{\lambda_n} = \widetilde{\psi} \text{ in } (C^{\alpha}(\overline{\Omega}))_{\mathbb{C}}.
\]

Since $\lim_{n \to \infty} u_{\lambda_n} = 0$ in $C^{\alpha}(\overline{\Omega})$, we have:
\[
\lim_{n \to \infty} M_{\lambda_n} \psi_{\lambda_n} = \widetilde{\psi} - i a \widetilde{\psi} \text{ in } (C^{\alpha}(\overline{\Omega}))_{\mathbb{C}}.
\]

By the continuity of $(L_{\lambda_{*}} - I)^{-1}$ from $(C^{\alpha}(\overline{\Omega}))_{\mathbb{C}}$ to $(C^{2,\alpha}(\overline{\Omega}))_{\mathbb{C}}$,
\[
\lim_{n \to \infty} \psi_{\lambda_n} = -\lim_{n \to \infty} (L_{\lambda_{*}} - I)^{-1} M_{\lambda_n} \psi_{\lambda_n} = -(L_{\lambda_{*}} - I)^{-1} (\widetilde{\psi} - i a \widetilde{\psi}) \text{ in } (C^{2,\alpha}(\overline{\Omega}))_{\mathbb{C}}.
\]

Thus, $\lim_{n \to \infty} \psi_{\lambda_n} = \widetilde{\psi}$ in $(C^{2,\alpha}(\overline{\Omega}))_{\mathbb{C}}$, and $L_{\lambda_{*}} \widetilde{\psi} - i a \widetilde{\psi} = 0$. Since $\| \psi_{\lambda_n} \|_{X} = \| \phi \|_{X}$, $\widetilde{\psi} \neq 0$, so $i a$ is an eigenvalue of $L_{\lambda_{*}}$. As $0$ is the principal eigenvalue, $a = 0$ and $\widetilde{\psi} = c \phi$ with $|c| = 1$. From (3.5), $c = 1$, proving the claim.

\noindent\textbf{Step 3.} We prove that $\frac{\nu_\lambda}{|\lambda - \lambda_{*}|}$ is bounded for $\lambda \in \Lambda$. Assume, for contradiction, that there exists a sequence $\{\lambda_l\}_{l=1}^{\infty}$ such that:
\[
\lim_{l \to \infty} \lambda_l = \lambda_*, \quad \lim_{l \to \infty} \frac{\nu_{\lambda_l}}{|\lambda_l - \lambda_{*}|} = +\infty, \quad \lim_{l \to \infty} \theta_{\lambda_l} = \theta_0 \in [0,2\pi].
\]

Note that $\langle \phi^{*}, L_{\lambda_{*}} \psi_{\lambda_l} \rangle = \langle L^{*}_{\lambda_{*}} \phi^{*}, \psi_{\lambda_l} \rangle = 0$. Substituting $(\psi, \nu, \theta, \lambda) = (\psi_{\lambda_l}, \nu_{\lambda_l}, \theta_{\lambda_l},\\$$ \lambda_l)$ and $u_{\lambda_l} = (\lambda_l - \lambda_{*}) \beta_{\lambda_l} \phi + h(t_{\lambda_l} \phi(x), \lambda_l)$ into \eqref{eigen}, multiplying by $\phi^{*}$, and integrating over $\Omega$ gives:
\begin{align}
\frac{i \nu_{\lambda_l}}{\lambda_l - \lambda_{*}} \int_{\Omega} \psi_{\lambda_l} \phi^{*} dx &= \int_{\Omega} m(x) \psi_{\lambda_l} \phi^{*} dx + \lambda_l \int_{\Omega} m_1(\xi, \beta, \lambda) \psi_{\lambda_l} \phi^{*} dx \nonumber \\
&\quad + \lambda_l \int_{\Omega} f^{\prime}(x, u_{\lambda_l}) \left( \beta_{\lambda_l} \phi + \frac{h(t_{\lambda_l} \phi(x), \lambda_l)}{\lambda_l - \lambda_{*}} \right) \psi_{\lambda_l} \phi^{*} dx \, e^{-i \theta_{\lambda_l}},
\end{align}
where
\[
m_1(\xi, \beta, \lambda) = \begin{cases}
\frac{f(t_\lambda \phi(x) + h(t_\lambda \phi(x), \lambda)) - m(x)}{\lambda - \lambda_{*}}, & \text{if } \lambda \neq \lambda_{*}, \\
f^{\prime}_{u}(x, 0) \beta_{\lambda_{*}} \phi, & \text{if } \lambda = \lambda_{*}.
\end{cases}
\]

From Step 2 and Theorem 2.3, $\lim_{\lambda \to \lambda_{*}} \psi_\lambda = \phi$ in $(C^{2,\alpha}(\overline{\Omega}))_{\mathbb{C}}$,
\[
\lim_{\lambda \to \lambda_{*}} \beta_{\lambda} = \frac{\int_{\Omega} m(x) \phi \phi^{*} dx}{-\lambda_{*} \int_{\Omega} f^{\prime}_{u}(x, 0) \phi^2 \phi^{*} dx}, \quad \text{and} \quad \lim_{\lambda \to \lambda_{*}} \frac{h(t_{\lambda} \phi(x), \lambda)}{\lambda - \lambda_{*}} = 0 \text{ in } (C^{\alpha}(\overline{\Omega}))_{\mathbb{C}}.
\]

Taking the limit as $l \to \infty$ in the equation yields:
\[
\limsup_{l \to \infty} \frac{\nu_{\lambda_l}}{|\lambda_l - \lambda_{*}|} \leq \frac{3 \int_{\Omega} m(x) \phi \phi^{*} dx}{\int_{\Omega} \phi \phi^{*} dx},
\]
a contradiction. 
\end{proof}

Ignoring a scalar factor, we see that $\psi$ in \eqref{eigen} can be represented as follows:
\begin{equation}\label{2.30}
\begin{cases}
\psi = r \phi + w, & r \geq 0,\quad w \in (X_{1})_{\mathbb{C}}, \\
\|\psi\|_{X} = \|\phi\|_{X}.
\end{cases} 
\end{equation}

\begin{lemma}
If $0 \in \sigma(\mathcal A_\tau(\la))$ for all $(\tau,\lambda) \in \mathbb{R}_{+} \times \Lambda$, then $\ds\int_{\Om}f^{\prime}_{u}(x, 0)\phi^2\phi^{*}dx = 0$.
\end{lemma}
\begin{proof}
    Substituting $\psi_{\lambda_{n}^{\prime}}=r_{\lambda_{n}^{\prime}}\phi+w_{\lambda_{n}^{\prime}}$,$u_{\lambda_{n}^{\prime}}=t_{\lambda_{n}^{\prime}} \phi+h\left(t_{\lambda_{n}^{\prime}} \phi, \lambda_{n}^{\prime}\right)$ and $\mu =0$ into the equation of \eqref{triangle},where ${\lambda_{n}^{\prime}}$ satisfies $\lim_{n \to \infty} \lambda_{n}^{\prime} = \lambda_{*}$ , multiplying it by $\phi^{*}(x)$ and then integrating on $\Omega$, we have
\[
\frac{1}{\lambda-\lambda_{*}} \left\{
\begin{aligned}
&\int_{\Omega} \left[ (\lambda-\lambda_{*})\phi^{*}(x)\left[m(x)+\lambda m_{1}(\xi,\beta,\lambda)\right](r_{\lambda_{n}^{\prime}}\phi(x)+w_{\lambda_{n}^{\prime}}) \right. \\
&\left. \quad + \lambda t_{\lambda}\phi(x)f^{\prime}(x,u_{\lambda})(r_{\lambda_{n}^{\prime}}\phi(x)+w_{\lambda_{n}^{\prime}})\phi^{*}(x) \right] \mathrm{d}x + o(\lambda-\lambda_{*})
\end{aligned}
\right\} = 0
\]

Since $\lim_{n \to \infty} \left\|w_{\lambda_{n}^{\prime}}\right\|_{\mathbb{Y}_{\mathbb{C}}} = 0$, then $\lim_{n \to \infty} \left\|w_{\lambda_{n}^{\prime}}\right\|_{L_{1}} = 0$.
Thus,
$\ds\int_{\Om}f^{\prime}_{u}(x, 0)\phi^2\phi^{*}dx = 0$.
\end{proof}

Next, we intend to use the implicit function theorem to prove that \eqref{eigen} is solvable.

Substituting $u_\lambda(x) = t_\lambda \phi(x) + h(t_\lambda \phi(x), \lambda),$ and $\nu = |\lambda-\lambda_{*}|h$ into \eqref{eigen}, we see that $(\nu,\theta,\psi)$ solves \eqref{eigen}, where $\nu>0$, $\theta\in[0,2\pi)$ and $\psi\in X_{\mathbb{C}}\backslash\{0\}$ satisfying \eqref{2.30}, then
\begin{equation}\label{2.31}
\begin{aligned}
\relax [L_{\lambda_{*}} +(\lambda -\lambda _{*})m(x)+\lambda (f\left(x,u_{\lambda}\right)-m(x))-i|\lambda-\lambda_{*}|h](r \phi + w)+ \\\lambda (t_\lambda \phi(x) + h(t_\lambda \phi(x), \lambda)) f^{\prime}_u\left(x,u_{\lambda}\right)(r \phi + w) e^{-i\theta}=0    
\end{aligned}
\end{equation}

Multiplying \eqref{2.31} by $\phi^{*}$ and integrating on $\Omega$, we have
\begin{equation*}
\begin{aligned}
\int_{\Omega} \phi^{*}(x) [L_{\lambda_{*}} +(\lambda -\lambda _{*})m(x)+\lambda (f\left(x,u_{\lambda}\right)-m(x))-i|\lambda-\lambda_{*}|h](r \phi(x) + w)+\\ \lambda (t_\lambda \phi(x) + h(t_\lambda \phi(x), \lambda)) f^{\prime}_u\left(x,u_{\lambda}\right)(r \phi + w) e^{-i\theta}\phi^{*}(x) \mathrm{d} x  =0      
\end{aligned}
\end{equation*}

Hence, if $\lambda \in \Lambda_{2_\epsilon}$,As $\lambda\to\lambda_{*}$, we have
\begin{equation*}\begin{aligned}
\int_{\Omega} \phi^{*}(x) [ m(x)+\lambda f^{\prime}_{u}(x, 0) \beta_\lambda  \phi-ih] \phi(x) +\lambda \beta_\lambda   f^{\prime}\left(x,0\right) \phi^{2}(x)  e^{-i\theta}\phi^{*}(x) \mathrm{d} x  =0      
\end{aligned}       
\end{equation*}

Separating the real and imaginary parts of the above formula, we get
$$
\left\{
\begin{array}{l}
\displaystyle
-\lambda_{*} \beta_{\lambda_{*}} \int_{\Omega} f^{\prime}\left(x,0\right) \phi^{2} \phi^{*} dx \sin \theta = h \int_{\Omega} \phi \phi^{*} dx \\
\displaystyle
\lambda_{*} \beta_{\lambda_{*}} \int_{\Omega} f^{\prime}\left(x,0\right) \phi^{2} \phi^{*} dx \cos \theta = 0
\end{array}
\right.
$$
combined with \eqref{al} implies that
\begin{equation}\label{2.32}
\theta  = \frac{\pi}{2}:=\theta_2, \quad h = h_{\lambda_{*}} = \frac{\int_{\Omega} m(x) \phi \phi^{*}  dx}{\int_{\Omega} \phi \phi^{*}  dx}  
\end{equation}

if $\lambda \in \Lambda_{1_\epsilon}$, we get:
\begin{equation}
\theta = \frac{3\pi}{2}:=\theta_1 , \quad h = h_{\lambda_{*}} = \frac{\int_{\Omega} m(x) \phi \phi^{*}  dx}{\int_{\Omega} \phi \phi^{*}  dx}  
\end{equation}

Without losing generality, we only consider that $\lambda \in \Lambda_{2_\epsilon}$, and the proof of the other part of the result is similar:
\begin{lemma}\label{2.7}
There exists a unique continuously differentiable mapping 
$\lambda\mapsto (w_{\lambda},r_{\lambda},h_{\lambda}$
$,\theta_{\lambda})$ 
from $\lambda\in\Lambda_{2\varepsilon}$ to $(X_{1})_{\mathbb{C}}\times\mathbb{R}^{3}$ 
such that,
$G(w_\la,r_\la,r_\la,\theta_\la,\la)=0$.Moreover, if $\lambda\in\Lambda_{2\varepsilon}$, and $(z^{\lambda}, \beta^{\lambda}, h^{\lambda}, \theta^{\lambda}, \lambda)$ solves the equation   
\begin{equation}\label{3.6G} 
\begin{cases}
G(w,r,h,\theta,\la)=0\\
z\in (X_1)_{\mathbb{C}},\;h,\;r\ge0, \;\theta\in[0,2\pi)\\
\end{cases}
\end{equation}
with $h^{\lambda} > 0$ and $\theta^{\lambda} \in [0, 2\pi)$, then 
\[
(z^{\lambda}, \beta^{\lambda}, h^{\lambda}, \theta^{\lambda}) = (z_{\lambda}, \beta_{\lambda}, h_{\lambda}, \theta_{\lambda}).
\]

Here
 \begin{equation*}
 G(w, r, h, \theta, \lambda) : (X_1)_{\mathbb{C}} \times [0, +\infty) \times [0, +\infty) \times [0, 2\pi] \times \mathbb{R} \to (Y_1)_{\mathbb{C}} \times \mathbb{C} \times \mathbb{R}     
 \end{equation*}
is defined by
\begin{equation*}
G(w, r, h, \theta, \lambda) = (g_1, g_2, g_3)^T,    
\end{equation*}
where
\begin{equation*}
\begin{cases}
g_1(w, r, h, \theta, \lambda) &:= L_{\lambda_{*}}w + (\lambda -\lambda _{*})[ m(x)+\lambda m_{1}(\xi, \beta, \lambda)-ih](r \phi + w) \\&+\lambda (t_\lambda \phi(x) + h(t_\lambda \phi(x), \lambda)) f^{\prime}_u\left(x,u_{\lambda}\right)(r \phi + w) e^{-i\theta}  \\& - (\lambda - \lambda_*)g_2(z, r, h, \theta, \lambda), \\
g_2(w, r, h, \theta, \lambda) &:= \int_{\Omega} \phi^{*}(x) [m(x)+\lambda m_{1}(\xi, \beta, \lambda)-ih](r \phi(x) + w)\\&+ \lambda \beta_\lambda \phi(x)  f^{\prime}_u\left(x,u_{\lambda}\right)(r \phi + w) e^{-i\theta}\phi^{*}(x) \mathrm{d} x\\    
g_3(w, r) &:= \|r\phi + w\|_X^2 - \|\phi\|_X^2.
\end{cases}
\end{equation*}

 \begin{proof}
 Obviously, $G(w_{\lambda_*}, r_{\lambda_*}, h_{\lambda_*}, \theta_{\lambda_*}, \lambda_{*})=G(0, 1, h_{\la_*},\pi/2, \lambda_{*})=0$,
Denote by $T = (T_1, T_2, T_3)^T$ the Fréchet derivative of $G$ with respect to $(w, r, h, \theta)$ at $(0, 1, h_{\la_*},\pi/2, \lambda_{*})$. A direct computation implies that $T : (X_1)_{\mathbb{C}} \times \mathbb{R}^3 \to (Y_1)_{\mathbb{C}} \times \mathbb{C} \times \mathbb{R}$ is given by
\begin{equation*}
\begin{cases}
T_1(\chi, \kappa, \zeta, \vartheta) &= L_{\lambda_{*}}\chi, \\
T_2(\chi, \kappa, \zeta, \vartheta) &= \displaystyle\int_{\Omega} [m(x)\phi^* + \lambda_*\beta_{\lambda_*}f^{\prime}\left(x,0\right)\phi\phi^* - \mathrm{i}\lambda_*\beta_{\lambda_*}f^{\prime}\left(x,0\right)\phi\phi^* \\&- \mathrm{i}h_{\lambda_*}\phi^*](\kappa\phi + \chi)  dx 
 - \mathrm{i}\zeta \int_{\Omega} \phi\phi^*  dx \\&- \lambda_*\beta_{\lambda_*}\vartheta \int_{\Omega} f^{\prime}\left(x,0\right)\phi^2\phi^*  dx, \\
T_3(\chi, \kappa, \zeta, \vartheta) &= \langle \phi, \chi \rangle_X + \langle \chi, \phi \rangle_X + 2\kappa\|\phi\|_X^2.
\end{cases}
\end{equation*}

Since $T$ is a bijection from $(X_1)_{\mathbb{C}} \times \mathbb{R}^3$ to $(Y_1)_{\mathbb{C}} \times \mathbb{C} \times \mathbb{R}$, the implicit function theorem guarantees the existence of a continuously differentiable map $\lambda \mapsto (w_\lambda, r_\lambda, h_\lambda, \theta_\lambda)$ from $\Lambda_{2\epsilon}$ to $(X_1)_{\mathbb{C}} \times \mathbb{R}^{3}$ such that $(w_\lambda, r_\lambda, h_\lambda, \theta_\lambda)$ solves \eqref{3.6G}.

To establish uniqueness, suppose $(w^\lambda, r^\lambda, h^\lambda, \theta^\lambda)$ is another solution of (3.12). We show that  
\[
(w^\lambda, r^\lambda, h^\lambda, \theta^\lambda) \to (w_{\lambda_*}, r_{\lambda_*}, h_{\lambda_*}, \theta_{\lambda_*})
\]  
as $\lambda \to \lambda_*$ in $X_{\mathbb{C}} \times \mathbb{R}^{3}$. By Lemma \ref{nu},  
\[
\lim_{\lambda \to \lambda_*} r^\lambda = r_{\lambda_*} = 1 \quad \text{and} \quad \lim_{\lambda \to \lambda_*} w^\lambda = w_{\lambda_*} = 0 \quad \text{in } (C^{2,\alpha}(\overline{\Omega}))_{\mathbb{C}}.
\] 

The sequences $\{h^\lambda\}$ and $\{\theta^\lambda\}$ are bounded for $\lambda \in \Lambda_{2\epsilon}$. For any sequence $\{\lambda_n\}$ with $\lambda_n \to \lambda_*$, there exists a subsequence $\{\lambda_{n_{k}}\}_{k=1}^{\infty}$ (denoted as $\{\lambda_{n}\}_{n=1}^{\infty}$ for simplicity) such that  
\[
\lim_{n \to \infty} h^{\lambda_n} = h^{\lambda_*} \quad \text{and} \quad \lim_{n \to \infty} \theta^{\lambda_n} = \theta^{\lambda_*}.
\] 

Taking the limit in $G(w^{\lambda_n}, r^{\lambda_n}, h^{\lambda_n}, \theta^{\lambda_n}) = 0$ yields $G(w_{\lambda_*}, r_{\lambda_*}, h^{\lambda_*}, \theta^{\lambda_*}, \lambda_*) = 0$. Combined with \eqref{2.32}, this implies  
\[
(h^{\lambda_*}, \theta^{\lambda_*}) = (h_{\lambda_*}, \theta_{\lambda_*}),
\]  

This completes the proof. 
\end{proof}

\end{lemma}

Noticing $(\nu, \theta, \psi)$ solves \eqref{eigen}, where $\nu > 0$, $\theta \in [0, 2\pi)$ and $\psi \in X_{\mathbb{C}}$ satisfying \eqref{2.30}, if and only if \eqref{3.6G}
admits a solution $(w, r, h, \theta)$. From lemma \ref{2.7}, we derive the following result.
\begin{theorem}\label{c25}
Assume that $\ds\int_{\Om}f^{\prime}_{u}(x, 0)\phi^2\phi^{*}dx \neq 0$, then there exists $\epsilon > 0$ such that

\begin{enumerate}
    \item[(i)]   for each $\lambda \in \Lambda_{2\epsilon}$ system \eqref{delay} has exactly one spatially nonhomogeneous steady-state solution, whose associated infinitesimal generator $\mathcal{A}_{\tau,\lambda}$ has an imaginary eigenvalue $\mathrm{i}\nu$ if and only if
\begin{equation*}\label{par}
\nu=\nu_\la=(\la-\la_*)h_\la,\;\psi= c \psi_\la,\;
\tau=\tau_{n}=\frac{\theta_\la+2n\pi}{\nu_\la},\;\; n=0,1,2,\cdots,
\end{equation*}
where $\psi_\la=r_\la\phi+w_\la$,
$c$ is a nonzero constant, and
$w_\la,r_\la,h_\la,\theta_\la$ are defined as in lemma \ref{2.7}, and $\lim_{\la \to \la_*} \theta_\la =\theta_2=\frac{\pi}{2}.$
    \item[(ii)]   for each $\lambda \in \Lambda_{1\epsilon}$ system \eqref{delay} has exactly one spatially nonhomogeneous steady-state solution, whose associated infinitesimal generator $\mathcal{A}_{\tau,\lambda}$ has an imaginary eigenvalue $\mathrm{i}\nu$ if and only if
\begin{equation*}
\nu=\nu_\la=-(\la-\la_*)h_\la,\;\psi= c \psi_\la,\;
\tau=\tau_{n}=\frac{\theta_\la+2n\pi}{\nu_\la},\;\; n=0,1,2,\cdots,
\end{equation*}
where $\psi_\la, w_\la, r_\la, h_\la, \theta_\la$ can be defined in a similar way to Lemma~\ref{2.7}, 
but $\lim_{\la \to \la_*} \\$$\theta_\la$$ = \theta_1$$ = \frac{3\pi}{2}$.
\end{enumerate}
\end{theorem}

Using a similar argument as above, we have:
\begin{theorem}
Let $\widetilde{\Delta}(\lambda,\mathrm{i}\nu_{\lambda},\tau_{\lambda,n})$ be the adjoint operator of $\Delta(\lambda,\mathrm{i}\nu_{\lambda},\tau_{\lambda,n})$. Then
\begin{equation*}
\begin{split}
&\widetilde\Delta(\la,\mu,\tau)\widetilde\psi:=\\&\nabla \cdot[d(x) \nabla \widetilde\psi]+\vec{b}(x)\cdot \nabla\widetilde\psi +\la f\left(x,u_{\lambda}\right)\widetilde\psi+ \lambda u_{\lambda} f^{\prime}_u\left(x,u_{\lambda}\right)\widetilde\psi e^{-\mu\tau}+\mu\widetilde\psi.
\end{split}
\end{equation*}
and the kernel space $\mathscr{N}\left(\widetilde{\Delta}(\lambda,\mathrm{i}\nu_{\lambda},\tau_{\lambda,n})\right) = \operatorname{span}\left\{\widetilde{\psi}_{\lambda}\right\}$. Moreover, ignoring a scalar factor, $\widetilde{\psi}_{\lambda}$ can be represented as
\begin{equation}\label{2.37}
\begin{cases}
\widetilde{\psi}_{\lambda} = \widetilde{r}_{\lambda}\phi^{*} + \widetilde{w}_{\lambda}, & \widetilde{r}_{\lambda} \geq 0,\ \widetilde{w}_{\lambda} \in (\widetilde{X}_{1})_{\mathbb{C}}, \\
\|\widetilde{\psi}_{\lambda}\|_{X} = \|\phi^{*}\|_{X},
\end{cases}
\end{equation}
and satisfies $\lim_{\lambda\rightarrow\lambda_{*}}\widetilde{\psi}_{\lambda} = \phi^{*}$ in $(C^{2}(\overline{\Omega}))_{\mathbb{C}}$, where $\phi^{*}$ is defined in lemma \ref{2.5}.
\end{theorem}

Now, we give some estimates to prove the simplicity of $i\nu_\la$.

\begin{lemma}\label{thm34}
Assume that $\lambda \in \Lambda$. Then,
for $n=0,1,2,\cdots$, \begin{equation}\label{sn}
S_{n}(\lambda):=\left[\int_{\Omega} \psi_{\lambda} \overline{\widetilde{\psi}}_{\lambda} d x+\lambda \tau_{\lambda, n} e^{-\mathrm{i} \theta_{\lambda}} \int_{\Omega} f^{\prime}_u\left(x,u_{\lambda}\right) u_{\lambda} \psi_{\lambda} \overline{\widetilde{\psi}}_{\lambda} d x\right]\neq 0.
\end{equation}
where $\psi_{\la}$ is defined as
in Theorem \ref{c25}.
\end{lemma}
\begin{proof}
Assuming $\lambda \in \Lambda_{2\epsilon}$ without loss of generality(if $\lambda \in \Lambda_{1\epsilon}$, the proof is similar).
It follows from lemma \ref{2.7} and throrem \ref{c25} that $\theta_{\la}\to \pi/2$, $\tau_{\lambda,n}(\la-\la_*)\to (\ds\f{\pi}{2}+2n\pi)/h_{\la_*}$, $\psi_\la\to\phi$ in $X_{\mathbb{C}}$ as $\la\to\la_*$. This, combined with Eq. \eqref{al}, yields
\begin{equation}\label{eq31}
\begin{split}
&\lim_{\la\to\la_*}S_n(\la) \\
=&\int_{\Omega} \phi\phi^* dx-\ds\f{i\beta_{\la_*}\la_*}{h_{\la_*}}\left(\ds\f{\pi}{2}+2n\pi\right)\int_\Omega f^{\prime}_{u}(x, 0)\phi^2\phi^*dx\\
=&\left[1+i(\frac{\pi}{2}+2n\pi)\right]\int_{\Om}\phi\phi^*dx\ne0.
 \end{split}
\end{equation}

This completes the proof.
\end{proof}
Then we show that the imaginary eigenvalue i$\nu$ is simple.
\begin{theorem}\label{thm34a}
Assume that
$\lambda \in \Lambda$. Then $\mu=i\nu_\lambda$ is a simple
eigenvalue of $A_{\tau_{n}}$ for $n=0,1,2,\cdots$, where $i\nu_\la$ and $\tau_{\lambda,n}$ are defined as in Theorem \ref{c25}.
\end{theorem}
\begin{proof}
It follows from Theorem \ref{c25} that $\mathscr{N}[A_{\tau_{n}}
(\lambda)-i\nu_\lambda]=\text{Span}[e^{i\nu_\lambda\theta}\psi_\lambda]$, where $\theta\in[-\tau_{\lambda,n},0]$ and $\psi_\la$ is defined as in
Theorem \ref{c25}.
If
$\phi_1\in\mathscr{N}[A_{\tau_{n}}
(\lambda)-i\nu_\lambda]^2$,
then
$$
[A_{\tau_{n}}
(\lambda)-i\nu_\lambda]\phi_1\in\mathscr{N}[A_{\tau_{n}}(\lambda)-i\nu_\lambda]=
\text{Span}[e^{i\nu_\lambda\theta}\psi_\lambda].
$$

Therefore, there exists a constant $a$ such that
$$
[A_{\tau_{n}}
(\lambda)-i\nu_\lambda]\phi_1=ae^{i\nu_\lambda\theta}\psi_\lambda,
$$

which yields
\begin{equation}
 \label{eq32}
\begin{split}
\dot{\phi_1}(\theta)&=i\nu_\lambda\phi_1(\theta)+ae^{i\nu_\lambda\theta}\psi_\lambda,
\ \ \ \ \theta\in[-\tau_{n},0], \\
 \dot{\phi_1}(0)&=
 \nabla\cdot\left[d(x)\nabla\phi_{1}(0)-\vec{b}(x)\phi_{1}(0)\right]+\lambda f\left(x, u_{\lambda}\right)\phi_{1}(0)  \\
&\quad +\lambda f^{\prime}_u\left(x, u_{\lambda}\right)u_{\lambda}\phi_{1}(-\tau_{\lambda,n}). 
 \end{split}
 \end{equation}
 
From the first equation of Eq. \eqref{eq32}, we see that
\begin{equation}
\label{eq33}
\begin{split}
\phi_1(\theta)&=\phi_1(0)e^{i\nu_\lambda\theta}+a\theta
e^{i\nu_\lambda\theta}\psi_\lambda,\\
\dot{\phi_1}(0)&=i\nu_\lambda\phi_1(0)+a\psi_\lambda.
\end{split}
\end{equation}

Eq. \eqref{eq32} and Eq. \eqref{eq33} imply that
\begin{equation}\label{provesimple}
\begin{split}
\begin{aligned}
\Delta\left(\lambda,\mathrm{i}\nu_{\lambda},\tau_{\lambda, n}\right)\phi_{1}(0) = &\nabla\cdot\left[d(x)\nabla\phi_{1}(0)-\vec{b}(x)\phi_{1}(0)\right]+\lambda f\left(x, u_{\lambda}\right)u_{\lambda}\phi_{1}(0)\\
&+\lambda f^{\prime}_u\left(x, u_{\lambda}\right)u_{\lambda}u_{\lambda}\phi_{1}(0) e^{-\mathrm{i}\theta_{\lambda}}-\mathrm{i}\nu_{\lambda}\phi_{1}(0)\\
= &a\psi_{\lambda}+a\lambda\tau_{\lambda, n}f^{\prime}_u\left(x, u_{\lambda}\right)u_{\lambda}u_{\lambda}\psi_{\lambda} e^{-\mathrm{i}\theta_{\lambda}}.
\end{aligned}
\end{split}\end{equation}

Since $\widetilde{\Delta}\left(\lambda, \mathrm{i} \nu_{\lambda}, \tau_{\lambda, n}\right) \widetilde{\psi}_{\lambda}=0$. This, combined with Eq. \eqref{provesimple}, yields
\begin{equation*}\begin{split}
0 &= \left\langle\widetilde{\Delta}\left(\lambda, \mathrm{i} \nu_{\lambda}, \tau_{\lambda, n}\right) \widetilde{\psi}_{\lambda}, \phi_{1}(0)\right\rangle = \left\langle\widetilde{\psi}_{\lambda}, \Delta\left(\lambda, \mathrm{i} \nu_{\lambda}, \tau_{\lambda, n}\right) \phi_{1}(0)\right\rangle \\
  &= a S_{n}(\lambda)
\end{split}
\end{equation*} 
which implies that
$a=0$ from Lemma \ref{thm34}. Therefore,
$$
\mathscr{N}[A_{\tau_{n}}(\lambda)-i\nu_\lambda]^j
=\mathscr{N}[A_{\tau_{n}}(\lambda)-i\nu_\lambda],\;\;j=
2,3,\cdots,\;\; n=0,1,2,\cdots,
$$
and $\lambda=i\nu_\lambda$ is a simple eigenvalue of
$A_{\tau_{n}}$ for $n=0,1,2,\cdots.$
\end{proof}
\subsection{Stability and Hopf Bifurcation}

In what follows, we only consider the case where $\ds\int_{\Om}f^{\prime}_{u}(x, 0)\phi^2\phi^{*}dx \neq 0$. In this section, we first analyze the stability of the steady-state solution $u_{\lambda}$ of \eqref{delay} when $\tau = 0$.

\begin{proposition}\label{2.13}
If $\ds\int_{\Om}f^{\prime}_{u}(x, 0)\phi^2\phi^{*}dx \neq 0$, then there exists $\varepsilon > 0$ such that for each $\lambda \in \Lambda_{2\varepsilon}$, all the eigenvalues of $\mathcal{A}_{0,\lambda}$ have negative real parts, and hence, the steady-state solution $u_{\lambda}$ of \eqref{delay} with $\tau = 0$ and $\lambda \in \Lambda_{2\varepsilon}$ is locally asymptotically stable.
\end{proposition}

\begin{proof}
We assume that the conclusion is not true, that is, there exists a sequence $\{\lambda_{n}^{\prime}\}_{n=1}^{\infty} \subseteq \Lambda_{2\varepsilon}$, such that $\lim_{n \to \infty} \lambda_{n}^{\prime} = \lambda_{*}$, and $\lambda_{n}^{\prime} > \lambda_{*}$ for $n \geq 1$, the corresponding eigenvalue problem
\begin{equation}\label{2.43}
\begin{cases}
\mu\psi =
\nabla \cdot[d(x) \nabla \psi-\vec{b}(x) \psi]+\la
f\left(x,u_{\lambda}\right) \psi
+ \lambda u_{\lambda} f^{\prime}_u\left(x,u_{\lambda}\right) \psi,& x\in\Om,\; \\
\psi=0,&x\in\partial\Om,\; 
\end{cases}
\end{equation}
has an eigenvalue $\mu_{\lambda_{n}^{\prime}}$ with nonnegative real part and the associated eigenfunction $\psi_{\lambda_{n}^{\prime}}$ satisfying $\|\psi_{\lambda_{n}^{\prime}}\|_{\mathbb{Y}_{\mathbb{C}}}=1$. For each $n\geq 1$, we write $\psi_{\lambda_{n}^{\prime}}$ as $\psi_{\lambda_{n}^{\prime}}=r_{\lambda_{n}^{\prime}}\phi+w_{\lambda_{n}^{\prime}}$, where $r_{\lambda_{n}^{\prime}}\geq 0$ satisfies $\lim_{n\to\infty}r_{\lambda_{n}^{\prime}}=1$ and $\langle w_{\lambda_{n}^{\prime}},\phi\rangle=0$. Substituting $\psi_{\lambda_{n}^{\prime}}=r_{\lambda_{n}^{\prime}}\phi+w_{\lambda_{n}^{\prime}}$,$u_{\lambda_{n}^{\prime}}=t_{\lambda_{n}^{\prime}} \phi+h\left(t_{\lambda_{n}^{\prime}} \phi, \lambda_{n}^{\prime}\right)$ and $\mu = \mu_{\lambda_{n}^{\prime}}$ into the first equation of \eqref{2.43}, multiplying it by $\phi^{*}(x)$ and then integrating on $\Omega$, we have
\[
\frac{1}{\lambda-\lambda_{*}} \left\{
\begin{aligned}
&\int_{\Omega} \left[ (\lambda-\lambda_{*})\phi^{*}(x)\left[m(x)+\lambda m_{1}(\xi,\beta,\lambda)-\frac{\mu}{\lambda-\lambda_{*}}\right](r_{\lambda_{n}^{\prime}}\phi(x)+w_{\lambda_{n}^{\prime}}) \right. \\
&\left. \quad + \lambda t_{\lambda}\phi(x)f^{\prime}(x,u_{\lambda})(r_{\lambda_{n}^{\prime}}\phi(x)+w_{\lambda_{n}^{\prime}})\phi^{*}(x) \right] \mathrm{d}x + o(\lambda-\lambda_{*})
\end{aligned}
\right\} = 0
\]

Since $\lim_{n \to \infty} \left\|w_{\lambda_{n}^{\prime}}\right\|_{\mathbb{Y}_{\mathbb{C}}} = 0$, then $\lim_{n \to \infty} \left\|w_{\lambda_{n}^{\prime}}\right\|_{L_{1}} = 0$.
Thus,
\begin{equation*}
\lim_{n \to \infty} \frac{\mu_{\lambda_{n}^{\prime}}}{\lambda^* - \lambda_{n}^{\prime}} = h_{\lambda_{*}} > 0,    
\end{equation*}
which implies there exists $N_{*}\in N$ such that for each $n\ge N_{*}$
\begin{equation}
\mathcal{R}e\mu_{\lambda_{n}^{\prime}}<0. 
\end{equation}

This is a contradiction with $\mathcal{R}e(\mu_{\lambda_{n}^{\prime}})\ge 0$ for $n\ge 1$. Therefore, there exists $\varepsilon>0$ such that all the eigenvalues of $\mathcal{A}_{0,\lambda}$ have negative real parts when $\lambda\in\Lambda_{2\varepsilon}$. This completes the proof of Proposition \ref{2.13}.
\end{proof}
Using similar arguments to the proof of Proposition \ref{2.13}, we can obtain the following result:
\begin{proposition}
If  $\ds\int_{\Om}f^{\prime}_{u}(x, 0)\phi^2\phi^{*}dx \neq 0$ then there exists $\varepsilon>0$ such that for each $\lambda\in\Lambda_{1\varepsilon}$, the infinitesimal generator $\mathcal{A}_{0,\lambda}$ has at least one eigenvalue with positive real parts, and hence, the steady-state solution $u_{\lambda}$ of \eqref{delay} with $\tau=0$ and $\lambda\in\Lambda_{1\varepsilon}$ is locally asymptotically unstable.
\end{proposition}
Note that $\mu=i\nu_{\la}$ is a simple eigenvalue of $A_{\tau_{n}}$. It follows from
the implicit function theorem that there
are a neighborhood $O_{n}\times D_{n}\times
H_{n}\subset\mathbb{R}\times\mathbb{C}\times X_{\mathbb{C}}$ of
$(\tau_{n},i\nu_\lambda,\psi_\lambda)$ and a continuously
differential function $(\mu(\tau),\psi(\tau)):O_{n}\rightarrow D_{n}\times
H_{n}$ such that for each $\tau\in O_{n}$, the only eigenvalue of
$A_\tau(\lambda)$ in $D_{n}$ is $\mu(\tau),$ and
\begin{equation}\label{eq34}
\begin{aligned}
\Delta(\lambda,\mu(\tau),\tau)\psi(\tau) 
&= \nabla\cdot[d(x)\nabla\psi(\tau)-\vec{b}(x)\psi(\tau)]+\lambda f(x,u_{\lambda})\psi(\tau) \\
&\quad +\lambda u_{\lambda}f^{\prime}(x,u_{\lambda})\psi(\tau)e^{-\mu(\tau)\tau}-\mu(\tau)\psi(\tau) \\
&= 0.
\end{aligned}
\end{equation}

Moreover, $
\mu(\tau_{n})=i\nu_\lambda$, and $\psi(\tau_{n})=\psi_\lambda$.
Then we have the following transversality
condition.
\begin{theorem}\label{thm35}

Assume $\ds\int_{\Om}f^{\prime}_{u}(x, 0)\phi^2\phi^{*}dx \neq 0$ and $\lambda \in \Lambda$, then for each fixed $n \in N$, there exist a neighborhood $O_{n} \times D_{n} \times H_{n} \subset \mathbb{R} \times \mathbb{C} \times \mathbb{X}_{\mathbb{C}}$ of $(\tau_{n,\lambda}, \mathrm{i}\nu_{\lambda}, \psi_{\lambda})$ and a continuously differentiable function $(\mu, \psi) : O_{n} \to D_{n} \times H_{n}$ satisfying $\mu(\tau_{n,\lambda}) = \mathrm{i}\nu_{\lambda}$ and $\psi(\tau_{n,\lambda}) = \psi_{\lambda}$ such that the only eigenvalue of $\mathcal{A}_{\tau,\lambda}$ in $D_{n}$ is $\mu(\tau)$, and eq.\eqref{eq34} holds. Moreover,

\[
\frac{\mathrm{d}\mathcal{R}e(\mu(\tau_{n,\lambda}))}{\mathrm{d}\tau} > 0, \quad n \in N.
\]
\end{theorem}
\begin{proof}
Differentiating Eq.(\ref{eq34}) with respect to $\tau$ at
$\tau=\tau_{n}$ yields
\begin{equation}\label{transcond}
\begin{split}
&\frac{d\mu(\tau_{n})}{d\tau}\left[-\psi_{\lambda}-\lambda\tau_{\lambda,n}f^{\prime}_u\left(x,u_{\lambda}\right)u_{\lambda}\psi_{\lambda}e^{-i\theta_{\lambda}}\right] \\
&\quad + \Delta(\lambda, i\nu_{\lambda},\tau_{n})\frac{d\psi(\tau_{n})}{d\tau}-i\nu_{\lambda}\lambda f^{\prime}_u\left(x,u_{\lambda}\right)u_{\lambda}\psi_{\lambda}e^{-i\theta_{\lambda}}=0.
\end{split}
\end{equation}

Note that \begin{equation}\label{aa}\left\langle \widetilde{\psi}_\la,\Delta(\la,i\nu_\la,\tau_{\lambda,n})\ds\f{d\psi(\tau_{\lambda,n})}{d\tau}\right\rangle=\left\langle \widetilde\Delta(\la,i\nu_\la,\tau_{\lambda,n})\widetilde{\psi}_\la, \ds\f{d\psi(\tau_{\lambda,n})}{d\tau}\right\rangle=0.
\end{equation}

Then, multiplying Eq. \eqref{transcond} by $\overline{\widetilde\psi_\lambda}$ and
integrating the result over $\Om$, we have
\begin{equation*}
\label{eq35}
\begin{aligned}
\frac{\mathrm{d}\mu(\tau_{n,\lambda})}{\mathrm{d}\tau} 
&= -\frac{\left\langle\tilde{\psi}_{\lambda}, i\nu_{\lambda}\lambda f^{\prime}_u\left(x,u_{\lambda}\right)u_{\lambda}\psi_{\lambda}e^{-i\theta_{\lambda}}\right\rangle}
{\left\langle\tilde{\psi}_{\lambda}, \psi_{\lambda}+\lambda\tau_{\lambda,n}f^{\prime}_u\left(x,u_{\lambda}\right)u_{\lambda}\psi_{\lambda}e^{-i\theta_{\lambda}}\right\rangle} \\
&= -\frac{i\nu_{\lambda}\lambda e^{-\mathrm{i} \theta_{\lambda}} \int_{\Omega}  f^{\prime}_u\left(x,u_{\lambda}\right)u_{\lambda}\psi_{\lambda}(x) \overline{\tilde{\psi}}_{\lambda}(x)  \mathrm{d} x}
{\int_{\Omega} \overline{\tilde{\psi}}_{\lambda}(x) \psi_{\lambda}(x) \mathrm{d} x+\lambda \tau_{n, \lambda} e^{-\mathrm{i} \theta_{\lambda}} \int_{\Omega} f^{\prime}_u\left(x,u_{\lambda}\right)u_{\lambda}\psi_{\lambda}(x) \overline{\tilde{\psi}}_{\lambda}(x)  \mathrm{d} x} \\
&= -\frac{1}{\left|S_{n}(\lambda)\right|^{2}}\left[i\nu_{\lambda}\lambda e^{-\mathrm{i} \theta_{\lambda}} \iint_{\Omega \times \Omega} f^{\prime}_u\left(x,u_{\lambda}\right)u_{\lambda}\psi_{\lambda}(x) \overline{\tilde{\psi}}_{\lambda}(x) \tilde{\psi}_{\lambda}(y) \overline{\psi}_{\lambda}(y) \mathrm{d} x \mathrm{d} y\right. \\
&\quad \left.+i\nu_{\lambda}\lambda^{2} \tau_{n, \lambda} \left|\int_{\Omega} f^{\prime}_u\left(x,u_{\lambda}\right)u_{\lambda}\psi_{\lambda}(x) \overline{\tilde{\psi}}_{\lambda}(x) \mathrm{d} x\right|^{2}\right].
\end{aligned}
\end{equation*}

It follows from Eq. \eqref{2.32} and the expression of $u_\la$, $\theta_\la$, $\nu_\la$ and $\psi_\la$ that
\begin{equation*}
\lim_{\la\to\la_*}\ds\f{1}{(\la-\la_*)^2}\frac{d\mathcal{R}e[\mu(\tau_{n})]}{d\tau}
=\ds\f{h_{\la_*}^2}{\lim_{\la_\to\la_*}|{S_n}(\la)|^2}\left(\int_\Omega \phi\phi^*dx\right)^2>0.
\end{equation*}
\end{proof}

From Theorems \ref{c25}, \ref{thm34a}
and \ref{thm35}, we have the result on the distribution of eigenvalues of $A_\tau(\lambda)$.
\begin{theorem}\label{2.15}
If $\ds\int_{\Om}f^{\prime}_{u}(x, 0)\phi^2\phi^{*}dx \neq 0$, then there exists a positive constant $\varepsilon > 0$ such that
\begin{enumerate}
     \item[(i)] For each fixed $\lambda \in \Lambda_{2\varepsilon}$ , the infinitesimal generator $\mathcal{A}_{\tau,\lambda}$ has only eigenvalues with negative real parts,when $\tau \in [0, \tau_{0,\lambda})$,and exactly increasing $2n$ eigenvalues with positive real parts when $\tau \in (\tau_{n,\lambda}, \tau_{n+1,\lambda})$, $n=0,1,2,\cdots N$.
     \item[(ii)] For each fixed $\lambda \in \Lambda_{1\varepsilon}$, the infinitesimal generator $\mathcal{A}_{\tau,\lambda}$ has at least an eigenvalue with positive real part , when $\tau \in [0, \tau_{0,\lambda})$, and exactly increasing $2n$ eigenvalues with positive real parts when $\tau \in (\tau_{n,\lambda}, \tau_{n+1,\lambda})$, $n=0,1,2,\cdots N$.
\end{enumerate}
Thus, we have the following results on the stability and Hopf bifurcation of the steady-state solution $u_{\lambda}$.
\end{theorem}

\begin{theorem}\label{2.17}
If $\ds\int_{\Om}f^{\prime}_{u}(x, 0)\phi^2\phi^{*}dx \neq 0$, then there exists a positive constant $\varepsilon > 0$ such that
\begin{enumerate}
    \item[(i)] For $\lambda \in \Lambda_{2\varepsilon}$, the steady-state solution $u_{\lambda}$ of \eqref{delay} is locally asymptotically stable for $\tau \in [0, \tau_{0,\lambda})$ and unstable when $\tau \in (\tau_{0,\lambda}, \infty)$.Moreover, at $\tau = \tau_{n,\lambda}$ ($n=0,1,2,\cdots$) a Hopf bifurcation occurs and a branch of spatially nonhomogeneous periodic orbits of \eqref{delay} emerges from $(\tau_{n,\lambda}, u_{\lambda})$.
    \item[(ii)] For $\lambda \in \Lambda_{1\varepsilon}$, the steady-state solution $u_{\lambda}$ of \eqref{delay} is locally asymptotically unstable for all $\tau \in [0, \tau_{0,\lambda})$. Moreover, at $\tau = \tau_{n,\lambda}$ ($n=0,1,2,\cdots$) a Hopf bifurcation occurs and a branch of spatially nonhomogeneous periodic orbits of \eqref{delay} emerges from $(\tau_{n,\lambda}, u_{\lambda})$.
\end{enumerate}
\end{theorem}

Theorem \ref{2.17} implies that there exist $\varepsilon_{0}>0$ and a continuously differentiable function $[-\varepsilon_{0},\varepsilon_{0}]\mapsto(\tau_{n,\lambda}(\varepsilon),T_{n}(\varepsilon),u_{n}(\varepsilon,x,t))\in\mathbb{R}\times\mathbb{R}\times\mathbb{X}$ satisfying $\tau_{n,\lambda}(0)=\tau_{n,\lambda}$, $T_{n}(0)=\frac{2\pi}{\nu_{\lambda}}$, and $u_{n}(\varepsilon,t,x)$ is a $T_{n}(\varepsilon)$-periodic solution of \eqref{delay} such that $u_{n}(\varepsilon,t,x)=u_{\lambda}+\varepsilon v_{n}(\varepsilon,x,t)$ and $v_{n}(0,x,t)$ is a $\frac{2\pi}{\omega_{\lambda}}$-periodic of system \eqref{linear}. Moreover, there exists $\alpha>0$ such that if \eqref{delay} has a nonconstant periodic solution $u(t,x)$ of period $T$ for some $\tau>0$ with

\[
|\tau-\tau_{n,\lambda}|<\alpha,\quad\left|T-\frac{2\pi}{\omega_{\lambda}}\right|<\alpha,\quad\max_{t\in\mathbb{R},x\in\overline{\Omega}}|u(t,x)-u_{\lambda}(x)|<\alpha,
\]
then $\tau=\tau_{n,\lambda}(\varepsilon)$ and $u(t,x)=u_{n}(\varepsilon,t+\theta,x)$ for some $|\varepsilon|<\varepsilon_{0}$ and some $\theta\in\mathbb{R}$.

\section{The direction of the Hopf bifurcation}

In this section, we combine the
methods in
\cite{faria2001normal,faria2002smoothness,hassard1981theory,faria2002stability} to analyze the direction of the Hopf bifurcation of Eq. \eqref{delay}.
Letting $U(t)=u(\cdot,t)-u_\la$, $t=\tau_{\lambda,n}\tilde t$,
$\tau=\tau_{\lambda,n}+\gamma$, and dropping the tilde sign, system \eqref{delay} can be transformed as follows:
where $U_t\in \mathcal{C}=C([-1,0],Y)$, and
\begin{equation}\label{ab}
\frac{dU(t)}{dt}=\tau_{\lambda,n}\nabla\cdot[d(x)\nabla U(t)-\vec{b}(x)U(t)]+\tau_{\lambda,n}L_{0}(U_{t})+J(U_{t},\gamma), 
\end{equation}
where $U_{t}\in\mathcal{C}=C([-1,0],Y)$, and
\begin{align}
L_{0}(U_{t})&=\lambda f\left(x,u_{\lambda}\right)U(t)+\lambda f^{\prime}_u\left(x,u_{\lambda}\right)u_{\lambda}U(t-1),  \\
J(U_{t},\gamma)&=\gamma\nabla\cdot[d(x)\nabla U(t)-\vec{b}(x)U(t)]+\gamma L_{0}(U_{t}) \nonumber \\
&\quad +\lambda(\gamma+\tau_{\lambda,n})[f^{\prime}_u\left(x,u_{\lambda}\right)U(t)U(t-1)+(U(t)+u_{\lambda})R]\\
R&=f\left(x,U(t-1)+u_{\lambda}\right)-f\left(x,u_{\lambda}\right)-f^{\prime}_u\left(x,u_{\lambda}\right)U(t-1)
\end{align}

Then Eq. \eqref{ab} occurs Hopf bifurcation near the zero equilibrium  when $\gamma=0$. Let
$\mathcal {A}_{\tau_{\lambda,n}}$ be the infinitesimal generator of the
linearized equation
\begin{equation*}\label{lab}
\ds\frac{dU(t)}{dt}=\tau_{\lambda,n}\nabla\cdot[d(x)\nabla U(t)-\vec{b}(x)U(t)]+\tau_{\lambda,n}L_0(U_t).
\end{equation*}

It follows from \cite{wu1996theory} that \begin{equation*}\begin{split}
\begin{aligned}
\mathcal{A}_{\tau_{\lambda,n}} \Psi=&\dot\Psi,\\
\mathscr{D}\left(\mathcal{A}_{\tau_{\lambda,n}}\right)= &\left\{\Psi\in\mathcal{C}_{\mathbb{C}}\cap\mathcal{C}_{\mathbb{C}}^{1}:\Psi(0)\in X_{\mathbb{C}},\dot{\Psi}(0) =\tau_{\lambda,n} \nabla \cdot [d(x) \nabla \Psi(0) - \vec{b}(x) \Psi(0)] \right. \\
&\left.+\lambda\tau_{\lambda,n}f\left(x, u_{\lambda}\right)\Psi(0)+\lambda\tau_{\lambda,n}f^{\prime}_u\left(x, u_{\lambda}\right)\Psi(-1)\right\},
\end{aligned}
\end{split}\end{equation*}
where $\mathcal{C}^1_\mathbb{C}=C^1([-1,0],Y_\mathbb{C})$,
and Eq. \eqref{ab} can be written in the following abstract form
\begin{equation*}\label{abab}
\ds\frac{dU_t}{dt}=\mathcal{A}_{\tau_{\lambda,n}}U_t+X_0J(U_t,\gamma),
\end{equation*}
where
\begin{equation*}X_0(\theta)=\begin{cases}0, \;\;\; & \theta\in[-1,0),\\
I, \;\;\; &\theta=0.\\
\end{cases}\end{equation*}

It follows from Theorem \ref{2.15} that $\mathcal {A}_{\tau_{\lambda,n}}$ has
only one pair of purely imaginary eigenvalues $\pm i \nu_\la\tau_{\lambda,n}$,
which are simple, and the corresponding eigenfunction with respect to
$i\nu_\la\tau_{\lambda,n}$ (respectively, $-i\nu_\la\tau_{\lambda,n}$) is $\psi_\la e^{ i \nu_\la\tau_{\lambda,n}\theta}$
(respectively, $\overline{\psi_\la} e^{-i\nu_\la\tau_{\lambda,n}\theta}$) for $\theta\in[-1,0]$, where
$\psi_\la$ is defined as in Theorem \ref{c25}.

Following \cite{faria2002stability,gurney1980nicholson}, we introduce the formal
duality $\langle\langle\cdot,\cdot\rangle\rangle$ in $\mathcal{C}$  by
\begin{equation*}\label{bil}
\langle\langle\tilde\Psi,\Psi\rangle\rangle=\langle
\tilde\Psi(0),\Psi(0)\rangle+\la\tau_{\lambda,n}\int_{-1}^0\left\langle\tilde\Psi(s+1),u_\la f^{\prime}_u\left(x, u_{\lambda}\right)\Psi(s) \right\rangle ds,
\end{equation*}
for $\Psi\in \mathcal{C}_{\mathbb{C}}$ and $\tilde\Psi\in
\mathcal{C}_{\mathbb{C}}^*:= C([0,1],Y_{\mathbb{C}})$, 
As in \cite{Hale1971}, we can compute the
formal adjoint operator $\mathcal{A}^*_{\tau_{\lambda,n}}$ of $\mathcal{A}_{\tau_{\lambda,n}}$ with respect to the formal duality.

\begin{lemma}\label{dualoperator}
The
formal adjoint operator $\mathcal{A}^*_{\tau_{\lambda,n}}$ of $\mathcal{A}_{\tau_{\lambda,n}}$ is defined by
$$\mathcal{A}^*_{\tau_{\lambda,n}}\tilde\Psi(s)=-\dot{\tilde\Psi}(s),$$ and the domain
\begin{equation*}\begin{split}
\mathscr{D}\left(\mathcal{A^*}_{\tau_{\lambda,n}}\right)= &\left\{\tilde\Psi\in\mathcal{C^*}_{\mathbb{C}}\cap\mathcal{C^*}_{\mathbb{C}}^{1}:\tilde\Psi(0)\in X_{\mathbb{C}},-\dot{\tilde\Psi}(0) =\tau_{\lambda,n} \nabla \cdot [d(x) \nabla \tilde\Psi(0) + \vec{b}(x) \tilde\Psi(0)] \right. \\
&\left.+\lambda\tau_{\lambda,n}f\left(x, u_{\lambda}\right)\tilde\Psi(0)+\lambda\tau_{\lambda,n}f^{\prime}_u\left(x, u_{\lambda}\right)\tilde\Psi(1)\right\},
\end{split}\end{equation*}
where $(\mathcal{C}^*_\mathbb{C})^1=C^1([0,1],Y_\mathbb{C})$. Moreover,
$\mathcal{A}^*_{\tau_{\lambda,n}}$ and $\mathcal{A}_{\tau_{\lambda,n}}$ satisfy
\begin{equation*}\label{Atauadjont}
\langle\langle
\mathcal{A}^*_{\tau_{\lambda,n}}\tilde\Psi,\Psi\rangle\rangle=\langle\langle
\tilde\Psi,\mathcal{A}_{\tau_{\lambda,n}}\Psi\rangle\rangle\;\;\text{for
}\Psi\in\mathscr{D}(\mathcal{A}_{\tau_{\lambda,n}})\text{ and }\tilde\Psi\in
\mathscr{D}(\mathcal{A}^*_{\tau_{\lambda,n}}).
\end{equation*}
\end{lemma}
\begin{proof}
For $\Psi\in\mathscr{D}(\mathcal{A}_{\tau_{\lambda,n}})$ and $\tilde\Psi\in
\mathscr{D}(\mathcal{A}^*_{\tau_{\lambda,n}})$,
\begin{equation*}
\begin{split}
\langle\langle
\tilde\Psi,\mathcal{A}_{\tau_{\lambda,n}}\Psi\rangle\rangle=&\left\langle
\tilde\Psi(0),
(\mathcal{A}_{\tau_{\lambda,n}}\Psi)(0)\right\rangle+\la\tau_{\lambda,n}\int_{-1}^0\left\langle\tilde\Psi(s+1),
u_\la f^{\prime}_u\left(x, u_{\lambda}\right)\dot\Psi(s)\right\rangle ds\\
=&\left\langle\tilde\Psi(0),\tau_{\lambda,n} \nabla \cdot [d(x) \nabla \Psi(0) - \vec{b}(x) \Psi(0)]\right\rangle\\&+\la\tau_{\lambda,n}\left[\left\langle\tilde\Psi(s+1), u_\la f^{\prime}_u\left(x, u_{\lambda}\right)\Psi(s)\right\rangle_1\right]_{-1}^{0}\\+&\left\langle \tilde\Psi(0),\la \tau_{\lambda,n}f\left(x, u_{\lambda}\right) \Psi(0)+\la\tau_{\lambda,n} f^{\prime}_u\left(x, u_{\lambda}\right) u_{\la}\Psi(-1)\right\rangle\\
-&\la\tau_{\lambda,n}\int_{-1}^0\left\langle\dot{\tilde\Psi}(s+1),
u_\la f^{\prime}_u\left(x, u_{\lambda}\right)\Psi(s)\right\rangle ds\\
=&\left\langle
(\mathcal{A}^*_{\tau_{\lambda,n}}\tilde\Psi)(0),\Psi(0)\right\rangle
+\la\tau_{\lambda,n}\int_{-1}^0\left\langle-\dot{\tilde\Psi}(s+1),
u_\la f^{\prime}_u\left(x, u_{\lambda}\right)\Psi(s)\right\rangle ds\\
=&\langle\langle
\mathcal{A}^*_{\tau_{\lambda,n}}\tilde\Psi,\Psi\rangle\rangle.
\end{split}
\end{equation*}
\end{proof}

Similarly, Theorem \ref{2.15} implies that the operator $\mathcal{A}^*_{\tau_{\lambda,n}}$ possesses a unique pair of simple, purely imaginary eigenvalues $\pm i \nu_{\lambda}\tau_{\lambda,n}$. The associated eigenfunctions are given by $\overline{\psi}_{\lambda} e^{i \nu_{\lambda}\tau_{\lambda,n} s}$ and $\psi_{\lambda} e^{-i \nu_{\lambda}\tau_{\lambda,n} s}$ for $s \in [0,1]$, corresponding to $-i\nu_{\lambda}\tau_{\lambda,n}$ and $i\nu_{\lambda}\tau_{\lambda,n}$, respectively, where $\psi_{\lambda}$ is defined in Theorem \ref{c25}.

Following \cite{wu1996theory}, the center subspace of Eq. \eqref{ab} is $P = \operatorname{span}\{p(\theta), \overline{p}(\theta)\}$, with $p(\theta) = \psi_{\lambda} e^{i \nu_{\lambda}\tau_{\lambda,n}\theta}$ being the eigenfunction of $\mathcal{A}_{\tau_{\lambda,n}}$ corresponding to $i\nu_{\lambda}\tau_{\lambda,n}$. Its formal adjoint subspace is $P^* = \operatorname{span}\{q(s), \overline{q}(s)\}$, where $q(s) = \widetilde{\psi}_{\lambda} e^{i \nu_{\lambda}\tau_{\lambda,n} s}$ is the eigenfunction of $\mathcal{A}^*_{\tau_{\lambda,n}}$ corresponding to $-i\nu_{\lambda}\tau_{\lambda,n}$, and $\widetilde{\psi}_{\lambda}$ is defined in \eqref{2.37}.

Define $\Phi_p = (p(\theta), \overline{p}(\theta))$ and $\Psi_P = \frac{1}{\overline{S_n}(\lambda)}(q(s), \overline{q}(s))^{T}$, where $S_n(\lambda)$ is given in Lemma \ref{thm34}. One can verify that $\langle \langle \Psi_p, \Phi_p \rangle \rangle = I$, where $I$ is the $2 \times 2$ identity matrix. Furthermore, the space $\mathcal{C}_{\mathbb{C}}$ decomposes as $\mathcal{C}_{\mathbb{C}} = P \oplus Q$, where
\[
Q = \{ \Psi \in \mathcal{C}_{\mathbb{C}} : \langle \langle \tilde{\Psi}, \Psi \rangle \rangle = 0 \text{ for all } \tilde{\Psi} \in P^* \}.
\]

Since the Hopf bifurcation formulas pertain only to the case $\gamma = 0$, we set $\gamma = 0$ in Eq. \eqref{ab}. Let the center manifold with range in $Q$ be given by
\begin{equation*} \label{center}
w(z, \overline{z}) = w_{20}(\theta) \frac{z^{2}}{2} + w_{11}(\theta) z \overline{z} + w_{02}(\theta) \frac{\overline{z}^{2}}{2} + \cdots.
\end{equation*}

The flow of Eq. \eqref{ab} on this center manifold can then be expressed as:

\begin{equation*}U_t=\Phi_p\cdot (z(t),\overline z(t))^{T}+w(z(t),\overline z(t)),\end{equation*}
where
\begin{equation*}\label{z(t)}
\begin{split}
\dot{z}(t) =&\ds\f{d}{dt}\langle\langle q(s),U_t\rangle\rangle\\
=&\langle\langle q(s),
\mathcal{A}_{\tau_{\lambda,n}}U_t\rangle\rangle+\ds\f{1}{S_n(\la)}\langle\langle q(s), X_0J(U_t,0)\rangle\rangle\\
 =& i\nu_\la\tau_{\lambda,n} z(t)+\ds\f{1}{S_n(\la)}\left\langle q(0),
J\left(\Phi_p(z(t),\overline z(t))^{T}+w(z(t),\overline
z(t)),0\right)\right\rangle_1\\
=&i\nu_\la\tau_{\lambda,n}z(t)+g(z,\overline z).
\end{split}
\end{equation*}

Then,
\begin{equation*} \label{g}
\begin{split}
g(z,\overline z)=&\ds\f{1}{S_n(\la)}\left\langle q(0),
J\left(\Phi_p(z(t),\overline z(t))^{T}+w(z(t),\overline
z(t)),0\right)\right\rangle_1\\=&g_{20}\ds\frac{z^{2}}{2}+g_{11}z\overline
z+g_{02}\ds\frac{\overline z^{2}}{2}+g_{21}\ds\frac{z^{2}\overline
z}{2}+\cdots,
\end{split}
\end{equation*}
and an easy calculation implies that
\begin{equation}\label{gij}
\begin{aligned}
g_{20}^{\lambda} &= \frac{2\lambda\tau_{\lambda,n}}{S_{n}(\lambda)}  \int\limits_{\Omega} f^{\prime}_u\left(x, u_{\lambda}\right) \psi_{\lambda}^{2} \overline{\widetilde{\psi}}_{\lambda}e^{-\mathrm{i}\nu_{\lambda}\tau_{\lambda,n}} +\frac{1}{2}u_{\lambda}f^{\prime\prime}_{uu}\left(x, u_{\lambda}\right)\psi_{\lambda}^{2} \overline{\widetilde{\psi}}_{\lambda}e^{-\mathrm{2i}\nu_{\lambda}\tau_{\lambda,n}} dx, \\
g_{11}^{\lambda} &= \frac{\lambda\tau_{\lambda,n}}{S_{n}(\lambda)}  \int\limits_{\Omega} f^{\prime}_u\left(x, u_{\lambda}\right)(e^{\mathrm{i}\nu_{\lambda}\tau_{\lambda,n}} + e^{-\mathrm{i}\nu_{\lambda}\tau_{\lambda,n}}) |\psi_{\lambda}|^{2} \overline{\widetilde{\psi}}_{\lambda}   +u_{\lambda}f^{\prime\prime}_{uu}\left(x, u_{\lambda}\right)|\psi_{\lambda}|^{2}\overline{\widetilde{\psi}}_{\lambda}dx, \\
g_{02}^{\lambda} &= \frac{2\lambda\tau_{\lambda,n}}{S_{n}(\lambda)}  \int\limits_{\Omega} f^{\prime}_u\left(x, u_{\lambda}\right) \overline{\psi}_{\lambda}^{2}  \overline{\widetilde{\psi}}_{\lambda}e^{\mathrm{i}\nu_{\lambda}\tau_{\lambda,n}} +\frac{1}{2}u_{\lambda}f^{\prime\prime}_{uu}\left(x, u_{\lambda}\right)\overline{\psi}_{\lambda}^{2} \overline{\widetilde{\psi}}_{\lambda}e^{\mathrm{2i}\nu_{\lambda}\tau_{\lambda,n}} dx, \\ 
g_{21}^{\lambda} &= \frac{2\lambda\tau_{\lambda,n}}{S_{n}(\lambda)} \int\limits_{\Omega} \overline{\widetilde{\psi}}_{\lambda}[\frac{1}{2}u_{\lambda}f^{\prime\prime}_{uu}\left(x, u_{\lambda}\right) w_{20}(-1) \overline{\psi}_{\lambda}e^{\mathrm{i}\nu_{\lambda}\tau_{\lambda,n}}+ u_{\lambda}f^{\prime\prime}_{uu}\left(x, u_{\lambda}\right) \psi_{\lambda}e^{-\mathrm{i}\nu_{\lambda}\tau_{\lambda,n}} w_{11}(-1)  \\
&\quad+ f^{\prime}_u\left(x, u_{\lambda}\right) \left[ \psi_{\lambda} w_{11}(-1) + w_{11}(0) \psi_{\lambda}e^{-\mathrm{i}\nu_{\lambda}\tau_{\lambda,n}} + \frac{1}{2} w_{20}(0) \overline{\psi}_{\lambda}e^{\mathrm{i}\nu_{\lambda}\tau_{\lambda,n}} + \frac{1}{2} \overline{\psi}_{\lambda} w_{20}(-1) \right] \\
&\quad+ \frac{1}{2}f^{\prime\prime}_{uu}\left(x, u_{\lambda}\right) \left[ 2 \psi_{\lambda} \psi_{\lambda}\overline{\psi}_{\lambda} + \overline{\psi}_{\lambda} [\psi_{\lambda}e^{-\mathrm{i}\nu_{\lambda}\tau_{\lambda,n}}]^2 \right] \\
&\quad+\frac{1}{2}u_{\lambda}f^{\prime\prime\prime}_{uuu}\left(x, u_{\lambda}\right) [\psi_{\lambda}e^{-\mathrm{i}\nu_{\lambda}\tau_{\lambda,n}}]^2 \overline{\psi}_{\lambda}e^{\mathrm{i}\nu_{\lambda}\tau_{\lambda,n}}]  dx.
\end{aligned}
\end{equation}

To compute $g_{21}$, we need to compute $w_{20}(\theta)$
and $w_{11}(\theta)$ in the following.
As in \cite{Chen2012,hassard1981theory}, we see that $w_{20}(\theta)$
and $w_{11}(\theta)$ satisfy
\begin{equation}\label{ws}
\begin{cases}
(2i\nu_\la\tau_{\lambda,n}-\mathcal{A}_{\tau_{\lambda,n}})w_{20}=H_{20},\\
-\mathcal{A}_{\tau_{\lambda,n}}w_{11}=H_{11}.\\
\end{cases}
\end{equation}
Here, for $-1\le\theta<0$,
\begin{equation}\label{H20}
H_{20}(\theta)=-(g_{20}p(\theta)+\overline g_{02}\overline
p(\theta)),
\end{equation}
\begin{equation}\label{H11}
H_{11}(\theta)=-(g_{11}p(\theta)+\overline g_{11}\overline
p(\theta)),
\end{equation}

It follows that:
\begin{equation}\label{W20}w_{20}(\theta)=\ds\frac{ig_{20}}{\nu_\la\tau_{\lambda,n}}p(\theta)+
\ds\frac{i\overline g_{02}}{3\nu_\la\tau_{\lambda,n}}\overline
p(\theta)+E_{\lambda}e^{2i\nu_\la\tau_{\lambda,n}\theta},
\end{equation}
and

\begin{equation}\label{W11}
w_{11}(\theta)=-\ds\frac{ig_{11}}{\nu_\la\tau_{\lambda,n}}p(\theta)+\ds\frac{i\overline
g_{11}}{\nu_\la\tau_{\lambda,n}}\overline p(\theta)+F_{\lambda}.
\end{equation}

From Eq. \eqref{ws} with $\theta=0$, the definition
of $\mathcal{A}_{\tau_{\lambda,n}}$ and
we see that $E$ satisfies
\begin{equation*}
\Delta\left(\lambda, 2\mathrm{i}\nu_{\lambda},\tau_{\lambda, n}\right)E_{\lambda}=-2\lambda(f^{\prime}_u\left(x, u_{\lambda}\right) \psi_{\lambda}^{2} e^{-\mathrm{i}\nu_{\lambda}\tau_{\lambda,n}} +\frac{1}{2}u_{\lambda}f^{\prime\prime}_{uu}\left(x, u_{\lambda}\right)\psi_{\lambda}^{2} e^{-\mathrm{2i}\nu_{\lambda}\tau_{\lambda,n}}).
\end{equation*}

Note that $2i\nu_\la$ is not the
eigenvalue of $A_{\tau_{\lambda,n}}(\la)$ for $\la\in(\la_*,\tilde\la^*]$, and hence
\begin{equation}\label{E}
E_{\lambda}=-2\lambda\Delta\left(\lambda, 2\mathrm{i}\nu_{\lambda},\tau_{\lambda, n}\right)^{-1}(f^{\prime}_u\left(x, u_{\lambda}\right) \psi_{\lambda}^{2} e^{-\mathrm{i}\nu_{\lambda}\tau_{\lambda,n}} +\frac{1}{2}u_{\lambda}f^{\prime\prime}_{uu}\left(x, u_{\lambda}\right)\psi_{\lambda}^{2} e^{-\mathrm{2i}\nu_{\lambda}\tau_{\lambda,n}})
\end{equation}

Similarly, from Eqs. \eqref{ws}, \eqref{H11}, and \eqref{W11}, we have
\begin{equation}\label{F1}
F_{\lambda}=-\lambda\Delta\left(\lambda, 0,\tau_{\lambda, n}\right)^{-1}(f^{\prime}_u\left(x, u_{\lambda}\right)(e^{\mathrm{i}\nu_{\lambda}\tau_{\lambda,n}} + e^{-\mathrm{i}\nu_{\lambda}\tau_{\lambda,n}}) |\psi_{\lambda}|^{2}   +u_{\lambda}f^{\prime\prime}_{uu}\left(x, u_{\lambda}\right)|\psi_{\lambda}|^{2}) .
\end{equation}

Now, we compute the functions E and F.
\begin{lemma}\label{comf}
Assume that $E_{\lambda}$ and $F_{\lambda}$ satisfy \eqref{E} and \eqref{F1}, respectively, and $\lambda \in \Lambda_{2\epsilon}$.
Then
\begin{align*}
E_{\lambda} &= \frac{1}{\lambda - \lambda_{*}} (c_{\lambda} \phi + \eta), 
F_{\lambda} = \frac{\tilde{\eta}_{\lambda}}{\lambda - \lambda_{*}},
\end{align*}
where $\phi$ is defined in (2.7), $c_{\lambda}$ satisfies
\[
\lim_{\lambda \to \lambda_{*}} c_{\lambda} = \frac{2\mathrm{i}}{\beta_{\lambda_{*}}(2\mathrm{i} - 1)},
\]
and $\eta, \tilde{\eta}_{\lambda} \in (X_{1})_{\mathbb{C}}$ satisfy
\[
\lim_{\lambda \to \lambda_{*}} \|\eta\|_{X_{\mathbb{C}}} = 0, \quad 
\lim_{\lambda \to \lambda_{*}} \|\tilde{\eta}_{\lambda}\|_{X_{\mathbb{C}}} = 0.
\]
\end{lemma}
\begin{proof}
We just derive the estimate for $E_{\lambda}$, and that for $F_{\lambda}$ can be derived similarly. It follows from (3.18) that if $\psi$ satisfies
\begin{equation}\label{3.20}
\begin{aligned}
    \Delta(\lambda, 2\mathrm{i}\nu_{\lambda},\tau_{\lambda,n})\psi&=-2\lambda(\lambda-\lambda_{*})(f^{\prime}_u\left(x, u_{\lambda}\right) \psi_{\lambda}^{2} e^{-\mathrm{i}\nu_{\lambda}\tau_{\lambda,n}} \\&+\frac{1}{2}u_{\lambda}f^{\prime\prime}_{uu}\left(x, u_{\lambda}\right)\psi_{\lambda}^{2} e^{-\mathrm{2i}\nu_{\lambda}\tau_{\lambda,n}}).
\end{aligned}
\end{equation}
then $\psi=(\lambda-\lambda_{*})E_{\lambda}$. Substituting $u_{\lambda}(x)=t_{\lambda}\phi(x)+h(t_{\lambda}\phi(x),\lambda)$ and $\nu_{\lambda}=(\lambda-\lambda_{*})h_{\lambda}$ into \eqref{3.20}, we have
\begin{equation*}
\begin{aligned}
&L_{\lambda_{*}}\eta + (\lambda -\lambda _{*})[ m(x)+\lambda m_{1}(\xi, \beta, \lambda)-2ih](c \phi + \eta)+ \\&\lambda (t_\lambda \phi(x) + h(t_\lambda \phi(x), \lambda)) f^{\prime}_u\left(x,u_{\lambda}\right)(c \phi + \eta) e^{-2i\theta}\\&=-2\lambda(\lambda-\lambda_{*})(f^{\prime}_u\left(x, u_{\lambda}\right) \psi_{\lambda}^{2} e^{-\mathrm{i}\theta} +\frac{1}{2}u_{\lambda}f^{\prime\prime}_{uu}\left(x, u_{\lambda}\right)\psi_{\lambda}^{2} e^{-\mathrm{2i}\theta})  , \\    
\end{aligned}   
\end{equation*}

where $L_{\lambda_{*}}$ is defined in (2.3). We see that $\psi = c\phi + \eta$ solves \eqref{E}, where $c \in \mathbb{C}$ and $\eta \in (X_1)_{\mathbb{C}}$, if and only if $(c, \eta)$ solves the following system

\begin{equation*}
F(\eta, c, \lambda) = 0, 
\end{equation*}

where $F = (f_1, f_2)^T : (X_1)_{\mathbb{C}} \times \mathbb{C} \times \mathbb{R} \to (Y_1)_{\mathbb{C}} \times \mathbb{C}$, and

\begin{equation*}
\begin{cases}
f_1(\eta, c, \lambda) &:= L_{\lambda_{*}}\eta + (\lambda -\lambda _{*})[ m(x)+\lambda m_{1}(\xi, \beta, \lambda)-2ih](c \phi + \eta)+ \\&\lambda (t_\lambda \phi(x) + h(t_\lambda \phi(x), \lambda)) f^{\prime}_u\left(x,u_{\lambda}\right)(c \phi + \eta) e^{-2i\theta}\\&+2\lambda(\lambda-\lambda_{*})(f^{\prime}_u\left(x, u_{\lambda}\right) \psi_{\lambda}^{2} e^{-\mathrm{i}\theta} +\frac{1}{2}u_{\lambda}f^{\prime\prime}_{uu}\left(x, u_{\lambda}\right)\psi_{\lambda}^{2} e^{-\mathrm{2i}\theta}) \\&- (\lambda - \lambda_*)f_2(\eta, c, \lambda),\\
f_2(\eta, c, \lambda) &:= \int_\Omega [ m(x)+\lambda m_{1}(\xi, \beta, \lambda)-2ih](c \phi + \eta)\phi^*+ \\&\lambda (\beta_\lambda \phi(x) + \frac{h(t_\lambda \phi(x), \lambda)}{\lambda - \lambda_*}) f^{\prime}_u\left(x,u_{\lambda}\right)(c \phi + \eta) e^{-2i\theta}\phi^*\\&+2\lambda(f^{\prime}_u\left(x, u_{\lambda}\right) \psi_{\lambda}^{2} e^{-\mathrm{i}\theta} +\frac{1}{2}u_{\lambda}f^{\prime\prime}_{uu}\left(x, u_{\lambda}\right)\psi_{\lambda}^{2} e^{-\mathrm{2i}\theta})\phi^* dx.
\end{cases} 
\end{equation*}

We first solve (4.16) for $\lambda = \lambda_*$. Since $f_1(\eta, c, \lambda_*) = L\eta = 0$, we obtain that $\eta = \eta_{\lambda_*} = 0$. Noticing that
\begin{align*}
\beta_{\lambda_*} &= \ds\f{\ds\int_{\Om}m(x)\phi\phi^{*}dx}{-\la_*\ds\int_{\Om}f^{\prime}_{u}(x, 0)\phi^2\phi^{*}dx}, \psi_{\lambda_*}= \phi, 
\xi_{\lambda_*} = 0, 
\theta_{\lambda_*} = \frac{\pi}{2}, 
h_{\lambda_*} = \frac{\int_\Omega m(x)\phi\phi^* dx}{\int_\Omega \phi\phi^* dx},
\end{align*}

we see that $f_2(\eta_{\lambda_*}, c, \lambda_*) = 0$ if and only if
\[
c = c_{\lambda_*} = \frac{2\mathrm{i}}{\beta_{\lambda_*}(2\mathrm{i} - 1)}.
\]

A direct computation implies that the Fréchet derivative of $F$ with respect to $(\eta,c)$ at $(\eta_{\lambda_{*}},c_{\lambda_{*}},\lambda_{*})$ is as follows:
\[
\widetilde{T}=(\widetilde{T}_{1},\widetilde{T}_{2})^{T}:=(D_{(\eta,c)}f_{1}(\eta_{\lambda_{*}},c_{\lambda_{*}},\lambda_{*}),D_{(\eta,c)}f_{2}(\eta_{\lambda_{*}},c_{\lambda_{*}},\lambda_{*}))^{T},
\]

where
\[
\widetilde{T}_{1}[\varphi,\epsilon]=L_{\lambda_{*}}\varphi,
\]
and
\begin{equation*}
\begin{aligned}
\widetilde{T}_{2}[\varphi,\epsilon] = \int_{\Omega} \biggl[ & m(x) + \lambda m_{1}(\xi,\beta,\lambda) - 2\mathrm{i}h]\phi^{*} \\
& + \lambda\left(t_{\lambda}\phi(x) + h\left(t_{\lambda}\phi(x),\lambda\right)\right) f^{\prime}(x,u_{\lambda})e^{-2\mathrm{i}\theta}\phi^{*} \biggr] (\epsilon\phi+\varphi) \, dx.
\end{aligned}
\end{equation*}

Then $\widetilde{T}$ is a bijection from $(X_{1})_{\mathbb{C}}\times\mathbb{C}$ to $(Y_{1})_{\mathbb{C}}\times\mathbb{C}$. It follows from the implicit function theorem that there is  a continuously differentiable mapping $\lambda\mapsto(\eta_{\lambda},c_{\lambda})$ such that $F(\eta_{\lambda},c_{\lambda},\lambda)=0$ for $\lambda\in \Lambda_{2\epsilon}$, and $(\eta_{\lambda},c_{\lambda})=(\eta_{\lambda_{*}},c_{\lambda_{*}})$ for $\lambda=\lambda_{*}$. Therefore,
\[
\lim_{\lambda\to\lambda_{*}}c_{\lambda}=\frac{2\mathrm{i}}{\beta_{\lambda_{*}}(2\mathrm{i}-1)},\quad \text{and} \quad \lim_{\lambda\to\lambda_{*}}\eta_{\lambda}=0 \text{ in } X_{\mathbb{C}}.
\]
\end{proof}

Using similar arguments to the proof of lemma \ref{comf}, we can obtain the following result:
\begin{lemma}
Assume that $E_{\lambda}$ and $F_{\lambda}$ satisfy \eqref{E} and \eqref{F1}, respectively, and $\lambda \in \Lambda_{1\epsilon}$.
Then
\begin{align*}
E_{\lambda} &= \frac{1}{\lambda - \lambda_{*}} (c_{\lambda} \phi + \eta), 
F_{\lambda} = \frac{\tilde{\eta}_{\lambda}}{\lambda - \lambda_{*}},
\end{align*}
where $\phi$ is defined in (2.7), $c_{\lambda}$ satisfies
\[
\lim_{\lambda \to \lambda_{*}} c_{\lambda} = \frac{2\mathrm{i}}{\beta_{\lambda_{*}}(2\mathrm{i} + 1)},
\]
and $\eta, \tilde{\eta}_{\lambda} \in (X_{1})_{\mathbb{C}}$ satisfy
\[
\lim_{\lambda \to \lambda_{*}} \|\eta\|_{X_{\mathbb{C}}} = 0, \quad 
\lim_{\lambda \to \lambda_{*}} \|\tilde{\eta}_{\lambda}\|_{X_{\mathbb{C}}} = 0.
\]
\end{lemma}

It is well known that the following quantities determine the direction and stability of bifurcating periodic orbits (see \cite{hassard1981theory,wu1996theory,faria2002stability}):
\begin{align*}
C_1(0) &= \frac{i}{2\nu_{\lambda}\tau_n} \left( g_{11}g_{20} - 2|g_{11}|^2 - \frac{|g_{02}|^2}{3} \right) + \frac{g_{21}}{2}, \\
\mu_2 &= -\frac{\mathcal{R}e(C_1(0))}{\mathcal{R}e(\mu'(\tau_n))}, \\
\beta_2 &= 2\mathcal{R}e(C_1(0)), \\
T_2 &= -\frac{\mathcal{I}m(C_1(0)) + \mu_2\mathcal{I}m(\mu'(\tau_n))}{\tau_n}.
\end{align*}

Here
\begin{enumerate}
\item $\mu_2$ determines the direction of the Hopf bifurcation: if $\mu_2 > 0$ ($\mu_2 < 0$), then the bifurcating periodic solutions exist for $\tau > \tau_n$ ($\tau < \tau_n$), and the bifurcation is called forward (backward);
\item $\beta_2$ determines the stability of bifurcating periodic solutions: the bifurcating periodic solutions are orbitally asymptotically stable (unstable) if $\beta_2 < 0$ ($\beta_2 > 0$);
\item $T_2$ determines the period of the bifurcating periodic solutions: the period increases (decreases) if $T_2 > 0$ ($T_2 < 0$).
\end{enumerate}

Hence we have the following results:
\begin{theorem}\label{4.4}
If $\ds\int_{\Om}f^{\prime}_{u}(x, 0)\phi^2\phi^{*}dx \neq 0$, then there exists $\varepsilon>0$ such that
\begin{enumerate}
    \item for each $(n,\lambda)\in\mathbb{N}\cup\{0\}\times\Lambda_{2\varepsilon}$, a branch of spatially nonhomogeneous periodic orbits of \eqref{delay} emerges from $(\tau,u)=(\tau_{n,\lambda},u_{\lambda})$. Moreover, the direction of the Hopf bifurcation at $\tau=\tau_{n,\lambda}$ is forward and the bifurcating periodic solution from $\tau=\tau_{0,\lambda}$ is locally asymptotically stable.
    \item for each $(n,\lambda)\in\mathbb{N}\cup\{0\}\times\Lambda_{1\varepsilon}$, a branch of spatially nonhomogeneous periodic orbits of \eqref{delay} emerges from $(\tau,u)=(\tau_{n,\lambda},u_{\lambda})$. Moreover, the direction of the Hopf bifurcation at $\tau=\tau_{n,\lambda}$ is forward and the bifurcating periodic solution  is  unstable.
\end{enumerate}
\end{theorem}
\begin{proof}
    To begin with, we will give the proof of 1.
    From Eqs. (3.12), (3.16), (3.17) and (3.18), we can compute $g_{20}$, $g_{11}$, $g_{02}$ and $g_{21}$ for the periodic orbits emerging from the Hopf bifurcation of Eq. \ref{delay} obtained in Theorem \ref{2.17}. Since $\la \in \Lambda_{2\epsilon}$, $\lambda\rightarrow\lambda_{*}$

\[
\psi_{\lambda} \rightarrow \phi,\widetilde{\psi}_{\lambda} \rightarrow \phi^*, \quad u_{\lambda} /(\lambda-\lambda_{*}) \rightarrow \beta_{\lambda_*} \phi \text{ in } C(\overline{\Omega}),
\]

\[
v_{\lambda} /(\lambda-\lambda_{*}) \rightarrow h_{\lambda_{*}}, \quad v_{\lambda} \tau_{n} \rightarrow \frac{\pi}{2}+2 n \pi.
\]
then
\begin{align*}
&\lim_{\lambda\to\lambda_{*}}S_{n}(\lambda)=\left[1+i\left(\frac{\pi}{2}+2n\pi\right)\right]\int_{\Omega}\phi\phi^{*}dx, 
\quad \lim_{\lambda\to\lambda_{*}}(\lambda-\lambda_{*})F=0, &\\
&\lim_{\lambda\to\lambda_{*}}\tau_{\lambda,n}(\lambda-\lambda_{*})=\frac{1}{h_{\lambda_{*}}}\left(\frac{\pi}{2}+2n\pi\right),
\quad \lim_{\lambda\to\lambda_{*}}\nu_{\lambda}\tau_{n}=\frac{\pi}{2}+2n\pi, &\\
&\lim_{\lambda\to\lambda_{*}}(\lambda-\lambda_{*})E=\frac{2i}{\beta_{\lambda_{*}}(2i-1)}\phi. &
\end{align*}

So we compute that
\begin{align*}
&\lim_{\lambda\to\lambda_{*}}(\lambda-\lambda_{*})g_{20}=\frac{2i(\pi+4n\pi)}{\beta_{\lambda_{*}}(2+i(\pi+4n\pi))}, &\\
&\lim_{\lambda\to\lambda_{*}}(\lambda-\lambda_{*})g_{11}=0, &\\
&\lim_{\lambda\to\lambda_{*}}(\lambda-\lambda_{*})g_{02}=\frac{-2i(\pi+4n\pi)}{\beta_{\lambda_{*}}(2+i(\pi+4n\pi))}, &\\
&\lim_{\lambda\to\lambda_{*}}(\lambda-\lambda_{*})^{2}g_{21}=\frac{2(\pi+4n\pi)(1-3i)}{\beta_{\lambda_{*}}^{2}(10+5i(\pi+4n\pi))}+\frac{8i\pi(1+4n)}{3\beta_{\lambda_{*}}^{2}|2+i(\pi+4n\pi)|^{2}}. &
\end{align*}

Then we can compute that $\lim_{\lambda\to\lambda_{*}}\mathcal{R}e((\lambda-\lambda_{*})^{2}g_{21})<0$ and $\lim_{\lambda\to\lambda_{*}}\mathcal{R}e((\lambda-\lambda_{*})^{2}C_{1}(0))<0$.

When $\lambda\in \Lambda_{1\epsilon}$, we have
\begin{align*}
&\lim_{\lambda\to\lambda_{*}}S_{n}(\lambda)=\left[1+i\left(\frac{3\pi}{2}+2n\pi\right)\right]\int_{\Omega}\phi\phi^{*}dx, 
\quad \lim_{\lambda\to\lambda_{*}}(\lambda-\lambda_{*})F=0, &\\
&\lim_{\lambda\to\lambda_{*}}\tau_{\lambda,n}(\lambda-\lambda_{*})=\frac{-\left(\frac{3\pi}{2}+2n\pi\right)}{h_{\lambda_{*}}},
\quad \lim_{\lambda\to\lambda_{*}}\nu_{\lambda}\tau_{n}=\frac{3\pi}{2}+2n\pi, &\\
&\lim_{\lambda\to\lambda_{*}}(\lambda-\lambda_{*})E=\frac{2i}{\beta_{\lambda_{*}}(2i+1)}\phi. \\
&\lim_{\lambda\to\lambda_{*}}(\lambda-\lambda_{*})g_{20}=\frac{2i(3\pi+4n\pi)}{\beta_{\lambda_{*}}(2+i(3\pi+4n\pi))}\\
&\lim_{\lambda\rightarrow\lambda_{*}}(\lambda-\lambda_{*})g_{11} = 0, \\
&\lim_{\lambda\rightarrow\lambda_{*}}(\lambda-\lambda_{*})g_{02} = \frac{-2i(3\pi+4n\pi)}{\beta_{\lambda_{*}}(2+i(3\pi+4n\pi))}, \\
&\lim_{\lambda\rightarrow\lambda_{*}}(\lambda-\lambda_{*})^{2}g_{21} = \frac{2(3\pi+4n\pi)(-1-3i)}{\beta_{\lambda_{*}}^{2}(10+5i(3\pi+4n\pi))} + \frac{8i\pi(3+4n)}{3\beta_{\lambda_{*}}^{2}|2+i(3\pi+4n\pi)|^{2}}.
\end{align*}

Then we can compute that 
$\lim_{\lambda \to \lambda_{*}} \mathcal{R}e((\lambda - \lambda_{*})^{2} g_{21}) < 0$, 
and 
$\lim_{\lambda \to \lambda_{*}} \mathcal{R}e((\lambda - \lambda_{*})^{2} C_{1}(0)) < 0$.
\end{proof}

\section{Numerical Simulation and Conclusions}
\subsection{Numerical Simulation}
We focus on numerical simulations of the food-limited model to validate our theoretical results, as logistic growth has been thoroughly analyzed elsewhere \cite{liu2022}. The model was originally proposed by Davidson and Gourley \cite{davidson2001effects}:

\begin{equation}
\frac{\partial v(x,t)}{\partial t} = -d \frac{\partial^2 v(x,t)}{\partial x^2} + r v(x,t) \frac{k - v(x,t-\tau)}{k + c v(x,t-\tau)}
\label{eq:base_model}
\end{equation}

where $r > 0$ is the growth rate, $k > 0$ the carrying capacity, $c > 0$ a constant, and $r/c$ represents the mass replacement at saturation. To incorporate population movement toward favorable habitats, we add an advection term:

\begin{equation*}
\frac{\partial v(x,t)}{\partial t} = D \nabla \cdot \left[ d(x) \nabla v(x,t) - \vec{b}(x) v(x,t) \right] + r v(x,t) \frac{k - v(x,t-\tau)}{k + c v(x,t-\tau)}
\label{eq:advection_model}
\end{equation*}

After a variable substitution, the model becomes:

\begin{equation}\label{5.3}
\begin{cases}
\displaystyle
u_t = \nabla \cdot [d(x) \nabla u - \vec{b}(x) u] + \lambda u(x,t) \frac{1 - u(x,t-\tau)}{1 + c u(x,t-\tau)} \\
u(0,t) =u(\pi,t) = 0
\end{cases}
\end{equation}
with the following initial value
\begin{equation}\label{5.4}
u(x,t)=\eta(x,t),\quad x\in[0,\pi],\quad t\in[-\tau,0]
\end{equation}
where $\eta\in{C}$, Let
\[
f(u)=\frac{1-u}{1+cu}.
\]

It is obvious that this model satisfies hypotheses $\mathbf{(H_1)}$ and $\mathbf{(H_2)}$, and $f^{\prime}(u)<0$ for $u\in[0,1]$. Applying Theorems \ref{2.1}, \ref{2.17}, \ref {4.4}, we have the following results:

\begin{enumerate}
    
    \item[(i)] For $\lambda \in \Lambda_{2\varepsilon}$, Eq. \eqref{5.3} has a positive steady state solution $u_{\lambda}$.
    
    \item[(ii)] For $\lambda \in \Lambda_{2\varepsilon}$, the steady-state solution $u_{\lambda}$ of \eqref{5.3} is locally asymptotically stable for $\tau \in [0, \tau_{0,\lambda})$ and unstable when $\tau \in (\tau_{0,\lambda}, \infty)$.Moreover, at $\tau = \tau_{n,\lambda}$ ($n=0,1,2,\cdots$) a Hopf bifurcation occurs and a branch of spatially nonhomogeneous periodic orbits of \eqref{5.3} emerges from $(\tau_{n,\lambda}, u_{\lambda})$.
    
    \item[(iii)] for each $(n,\lambda)\in\mathbb{N}\cup\{0\}\times\Lambda_{2\varepsilon}$, a branch of spatially nonhomogeneous periodic orbits of \eqref{5.3} emerges from $(\tau,u)=(\tau_{n,\lambda},u_{\lambda})$. Moreover, the direction of the Hopf bifurcation at $\tau=\tau_{n,\lambda}$ is forward and the bifurcating periodic solution from $\tau=\tau_{0,\lambda}$ is locally asymptotically stable.
\end{enumerate}

To illustrate the analytical results presented in the previous sections, we conduct numerical simulations using model \eqref{5.3}--\eqref{5.4} under the following setting:
\[\Omega = \left\{(x,y) \in \mathbb{R}^{2} \mid (x-\pi)^{2} + (y-\pi)^{2} = \pi^{2} \right\}, \quad d(x,y) = 1 + 0.1x + 0.1y\]
\[c = 0.5  ,\quad \eta(x,t) = 0.001,\quad\vec{b}(x,y) = \left( \frac{\cos x}{\sin x + 2}, \frac{\cos y}{\sin y + 2} \right)\]

we use $\lambda$ = 1.0. For $\lambda$ = 1.0 and smaller $\tau$ , $u_{\lambda}$ is stable as shown in Figure 1. (a),(b).we present cross-sectional views of the solution $u(x,t)$ along the custom line segment from (3.14,0.00)to (3.14,6.28), corresponding to a fixed spatial coordinate $x$=$\pi$ with $y$ varying from $0$ to $2\pi$. The $x$-axis of this graph represents time $t$, while the $y$-axis represents position $x$. The color bar indicates the amplitude of $u$.

\begin{figure}[htbp]
    \centering
    \begin{subfigure}[b]{0.45\textwidth}
        \centering
        \includegraphics[width=\textwidth]{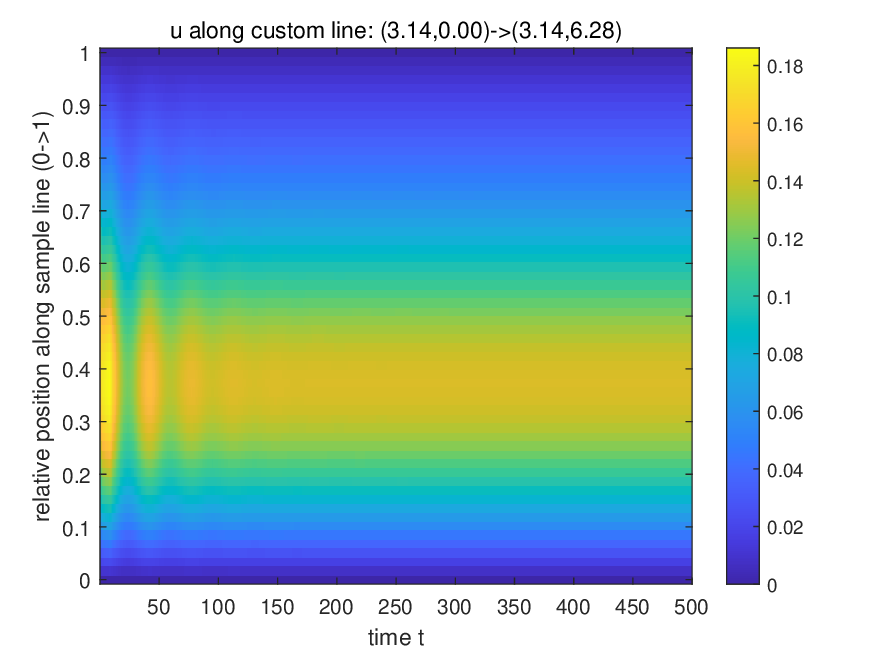}
        \caption{$\tau = 8$}
        \label{fig:tau8}
    \end{subfigure}
    \hfill
    \begin{subfigure}[b]{0.45\textwidth}
        \centering
        \includegraphics[width=\textwidth]{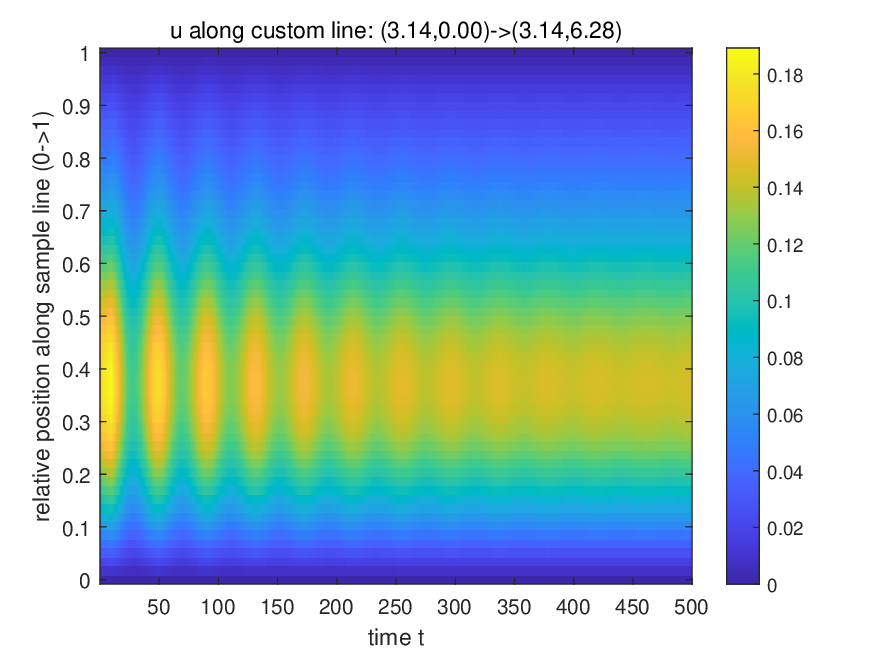}
        \caption{$\tau = 10$}
        \label{fig:tau10}
    \end{subfigure}
    
    \vspace{0.5cm} 
    
    \begin{subfigure}[b]{0.45\textwidth}
        \centering
        \includegraphics[width=\textwidth]{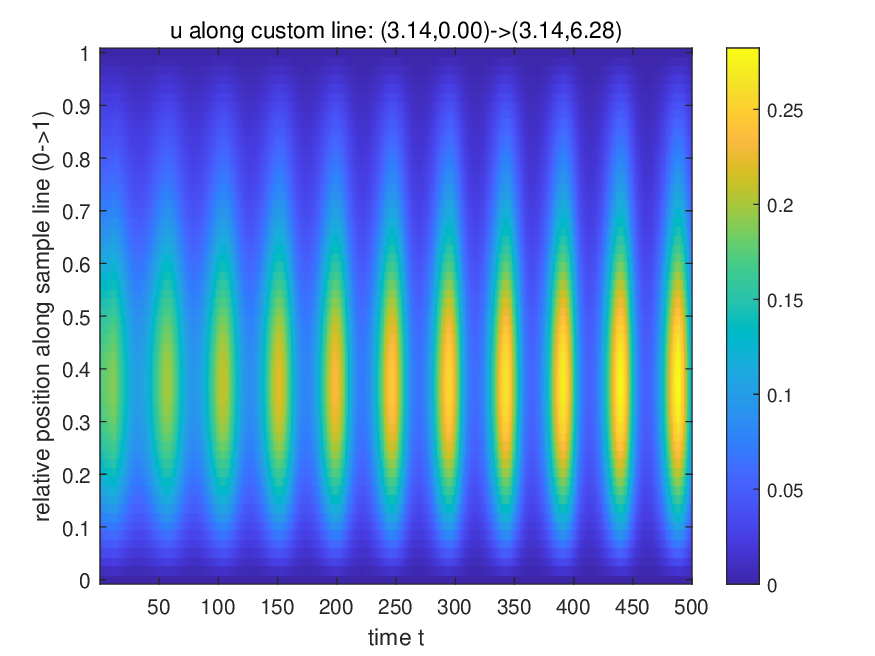}
        \caption{$\tau = 12$}
        \label{fig:tau12}
    \end{subfigure}
    \hfill
    \begin{subfigure}[b]{0.45\textwidth}
        \centering
        \includegraphics[width=\textwidth]{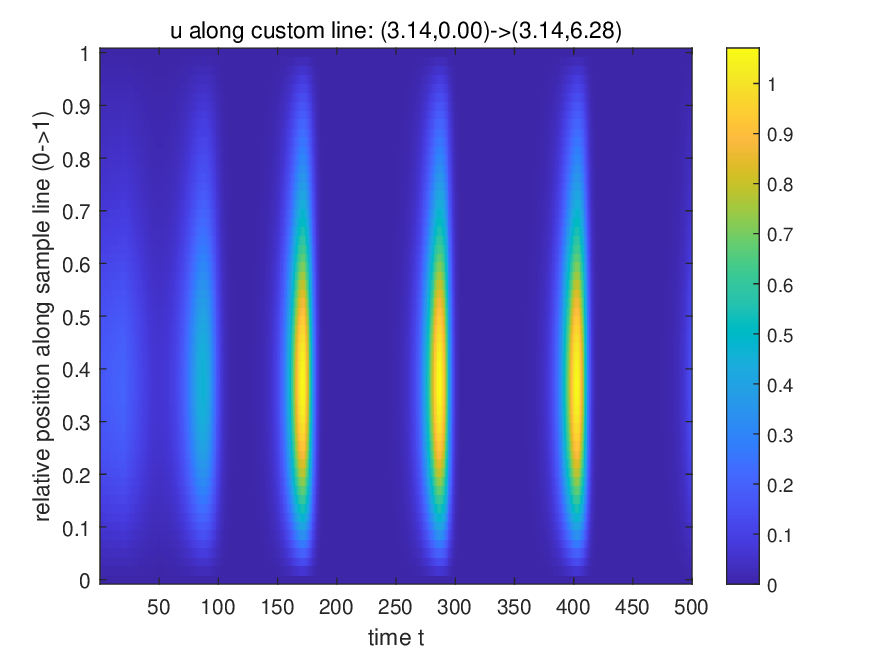}
        \caption{$\tau = 20$}
        \label{fig:tau20}
    \end{subfigure}
    
    \caption{
        (a) $\tau = 8$, the solution approaches to the positive steady state.
        (b) $\tau = 10$, the solution still approaches to the positive steady state but with noticeable oscillations. 
        (c) $\tau = 12$, the solution converges to a time-periodic solution with small oscillations.
        (d) $\tau = 20$, the solution converges to a time-periodic solution and the period and amplitude of the stable periodic solution both increase as the time delay $\tau$ increases.
    }
    \label{fig:time_delay_comparison}
\end{figure}
This phenomenon aligns with theoretical predictions where increasing $\tau$ drives the system toward criticality, beyond which a Hopf bifurcation would occur, resulting in stable periodic solutions, and the bifurcation is forward.

Next we discuss  model incorporating the weak Allee effect,Consider the following model with the same initial and boundary conditions as \eqref{5.3},\eqref{5.4}:
\begin{equation}
\begin{aligned}
&u_t = \nabla \cdot [d(x) \nabla u - \vec{b}(x) u] + 2\lambda u(x,t)(1-u(x,t-\tau))\left(u(x, t-\tau) + \frac{1}{2}\right), \\
\end{aligned}
\end{equation}

Note that the function $f(u) = 2(1 - u)(u + 0.5)$satisfies the fundamental hypothesis $\mathbf{(H_1)}$ , $\mathbf{(H_2)}$ and $f^{\prime}(0)>0$. Our previous analytical results demonstrate that for $\lambda \in \Lambda_{1\varepsilon}$, the positive steady state undergoes a Hopf bifurcation as the time delay $\tau$increases. However, since both the positive steady state and the bifurcating periodic solutions are unstable, these dynamics are unlikely to be observed in standard numerical simulations.

Interestingly, with appropriately chosen initial conditions, a similar stable periodic pattern emerges as $\tau$ increases for the weak Allee effect dynamics. As illustrated in Figure 2, we fix the parameter $\lambda = 0.8$ to systematically explore the impact of the time delay $\tau$ on the system's dynamics.Thus the Hopf bifurcation occurs within the interval $\tau$$\in$$[5.5,9.5]$.

We hypothesize that this phenomenon is attributed to the existence of an additional stable steady-state branch for $\lambda \in \Lambda_{1\varepsilon}$. This conjecture is consistent with the findings in \cite{su2009hopf}, but \cite{{su2009hopf}} does not incorporate a generalized advection term.

\begin{figure*}[htbp]
    \centering
    \begin{subfigure}[b]{0.45\textwidth}
        \centering
        \includegraphics[width=\textwidth]{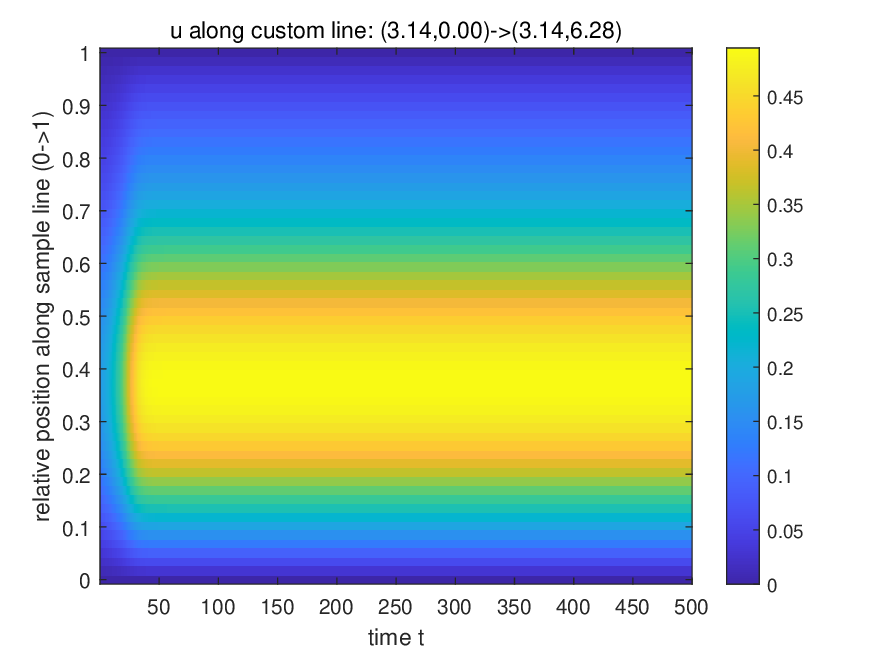}
        \caption{$\tau = 1.5$}
        \label{fig:tau15}
    \end{subfigure}
    \hfill
    \begin{subfigure}[b]{0.45\textwidth}
        \centering
        \includegraphics[width=\textwidth]{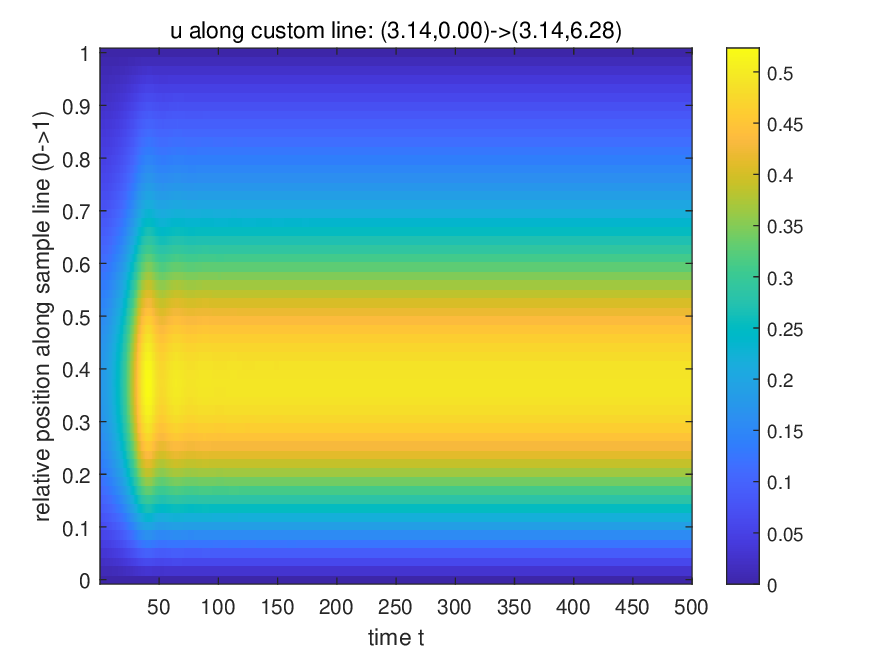}
        \caption{$\tau = 5.5$}
        \label{fig:tau55}
    \end{subfigure}
    \hfill
    \begin{subfigure}[b]{0.45\textwidth}
        \centering
        \includegraphics[width=\textwidth]{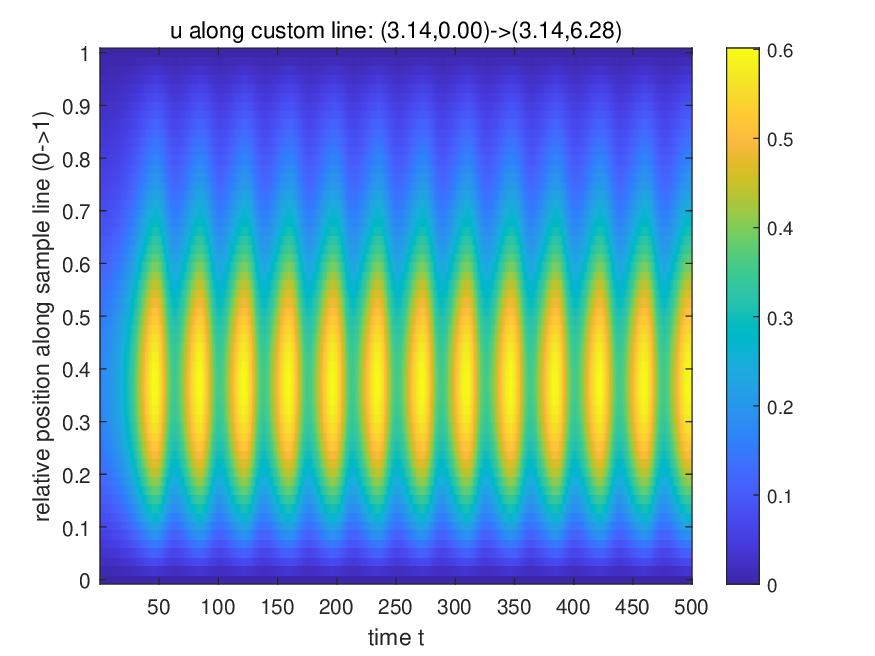}
        \caption{$\tau = 9.5$}
        \label{fig:tau95}
    \end{subfigure} \hfill
    \begin{subfigure}[b]{0.45\textwidth}
        \centering
        \includegraphics[width=\textwidth]{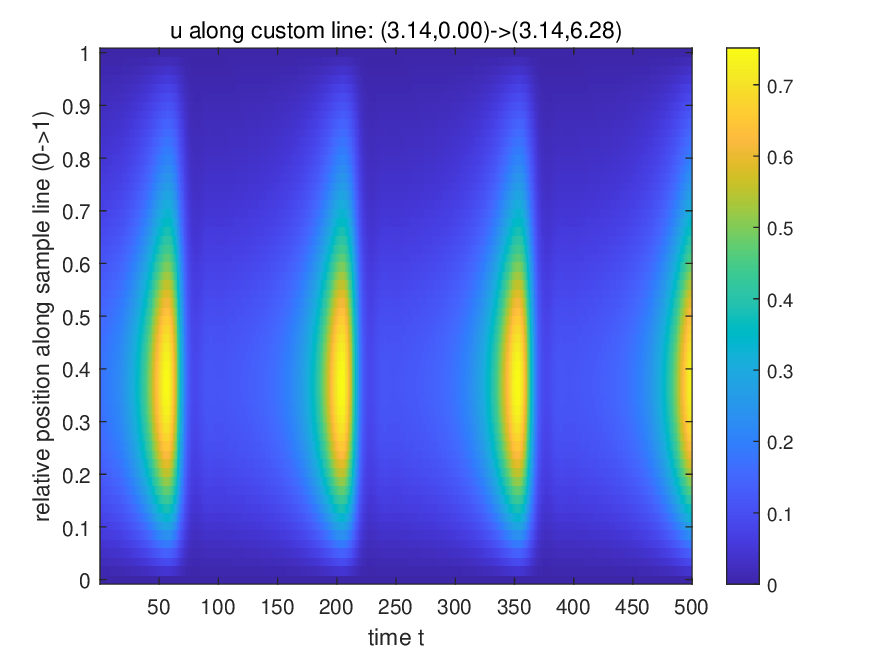}
        \caption{$\tau = 15.5$}
        \label{fig:tau155}
    \end{subfigure}

    \caption{(a) for $\tau$ = 1.5, the solution converges to positive steady state $u_\lambda$.(b) for $\tau$ = 5.5 , the solution approaches $u_\lambda$ with oscillations. (c)for $\tau$ = 9.5, the solution converges to a time-periodic solution.(d)for $\tau$ = 15.5the solution asymptotically converges to a time-periodic solution. Both the period and amplitude of the emerging stable periodic orbit increase with the time delay $\tau$. }
\end{figure*}
\newpage
\subsection{Conclusions}
In this paper, we have investigated the delay-induced instability in a reaction-diffusion model with a general advection term and a general time-delayed per capita growth rate. The primary focus was on the stability and Hopf bifurcation analysis of the positive steady state $u_\lambda$ when the parameter $\lambda$ is near the principal eigenvalue $\lambda_*$ of a non-self-adjoint elliptic operator. The specific points can be summarized as follows:

If $\ds\int_{\Om}f^{\prime}_{u}(x, 0)\phi^2\phi^{*}dx < 0$, then there exists $\varepsilon>0$ such that:
\begin{enumerate}
    \item[(i)] For $\lambda \in \Lambda_{2\varepsilon}$, Eq. \eqref{5.3} has a positive steady state solution $u_{\lambda}$.
    
    \item[(ii)] For $\lambda \in \Lambda_{2\varepsilon}$, the steady-state solution $u_{\lambda}$ of \eqref{5.3} is locally asymptotically stable for $\tau \in [0, \tau_{0,\lambda})$ and unstable when $\tau \in (\tau_{0,\lambda}, \infty)$.Moreover, at $\tau = \tau_{n,\lambda}$ ($n=0,1,2,\cdots$) a Hopf bifurcation occurs and a branch of spatially nonhomogeneous periodic orbits of \eqref{5.3} emerges from $(\tau_{n,\lambda}, u_{\lambda})$.
    
    \item[(iii)] for each $(n,\lambda)\in\mathbb{N}\cup\{0\}\times\Lambda_{2\varepsilon}$, a branch of spatially nonhomogeneous periodic orbits of \eqref{5.3} emerges from $(\tau,u)=(\tau_{n,\lambda},u_{\lambda})$. Moreover, the direction of the Hopf bifurcation at $\tau=\tau_{n,\lambda}$ is forward and the bifurcating periodic solution from $\tau=\tau_{0,\lambda}$ is locally asymptotically stable.
\end{enumerate}

If $\ds\int_{\Om}f^{\prime}_{u}(x, 0)\phi^2\phi^{*}dx > 0$, then there exists $\varepsilon>0$ such that:
\begin{enumerate}
    \item[(i)] For $\lambda \in \Lambda_{1\varepsilon}$, Eq. \eqref{5.3} has a positive steady state solution $u_{\lambda}$.
    
    \item[(ii)] For $\lambda \in \Lambda_{1\varepsilon}$, the steady-state solution $u_{\lambda}$ of \eqref{delay} is locally asymptotically unstable for all $\tau \in [0, \tau_{0,\lambda})$. Moreover, at $\tau = \tau_{n,\lambda}$ ($n=0,1,2,\cdots$) a Hopf bifurcation occurs and a branch of spatially nonhomogeneous periodic orbits of \eqref{delay} emerges from $(\tau_{n,\lambda}, u_{\lambda})$.
    
    \item[(iii)] for each $(n,\lambda)\in\mathbb{N}\cup\{0\}\times\Lambda_{1\varepsilon}$, a branch of spatially nonhomogeneous periodic orbits of \eqref{delay} emerges from $(\tau,u)=(\tau_{n,\lambda},u_{\lambda})$. Moreover, the direction of the Hopf bifurcation at $\tau=\tau_{n,\lambda}$ is forward and the bifurcating periodic solution  is  unstable.
\end{enumerate}


\newpage
\newpage
\begin{thebibliography}{10}


\bibitem{busenberg1996stability}
S.~Busenberg and W.~Huang.
\newblock Stability and {H}opf bifurcation for a population delay model with
  diffusion effects.
\newblock {\em  J. Differential Equations}, 124(1):80--107, 1996.



\bibitem{Chen2012}
S.~Chen and J.~Shi.
\newblock Stability and {H}opf bifurcation in a diffusive logistic population
  model with nonlocal delay effect.
\newblock {\em J. Differential Equations}, 253(12):3440--3470, 2012.

\bibitem{ChenYu2016}
S. Chen and J. Yu.
\newblock Stability and bifurcations in a nonlocal delayed reaction-diffusion population model.
\newblock \emph{J. Differential Equations}, 260(1):218--240, 2016.


\bibitem{liu2022}
J. Liu, S. Chen,
\newblock Delay-induced instability in a reaction-diffusion model with a general advection term.
\newblock J. Math. Anal. Appl,
\newblock vol. 512, no. 2, pp. 126--151, 2022.


\bibitem{Guo2016}
S.~Guo and S.~Yan.
\newblock Hopf bifurcation in a diffusive {L}otka-{V}olterra type system with
  nonlocal delay effect.
\newblock {\em J. Differential Equations}, 260(1):781--817, 2016.


\bibitem{Hale1971}
J.~Hale.
\newblock {\em Theory of {F}unctional {D}ifferential {E}quations}.
\newblock Springer-Verlag, New York, second edition, 1977.


\bibitem{yan2005}
Xiang-Ping Yan, Wan-Tong Li,
\newblock \textit{Stability of bifurcating periodic solutions in a delayed reaction--diffusion population model},
\newblock Nonlinearity,
\newblock Vol. 63, No. 5-7, pp. e373--e384, 2005.

\bibitem{su2009hopf}
Y.~Su, J.~Wei, and J.~Shi.
\newblock Hopf bifurcations in a reaction-diffusion population model with delay
  effect.
\newblock {\em J. Differential Equations}, 247(4):1156--1184, 2009.

\bibitem{wu1996theory}
J.~Wu.
\newblock {\em {T}heory and {A}pplications of {P}artial {F}unctional
  {D}ifferential {E}quations}.
\newblock Springer-Verlag, New York, 1996.


\bibitem{jin2021}
Z. Jin, R. Yuan,
\newblock Hopf bifurcation in a reaction-diffusion-advection equation with nonlocal delay effect,
\newblock \emph{J. Differential Equations},
\newblock vol. 271, pp. 176--223, 2021.

\bibitem{guo2016}
S. Guo, L. Ma,
\newblock \textit{Stability and Bifurcation in a Delayed Reaction--Diffusion Equation with Dirichlet Boundary Condition},
\newblock J. Nonlinear Sci, vol. 26, no. 2, pp. 545--580, 2016.

\bibitem{hu2011spatially}
R. Hu and Y. Yuan.
\newblock Spatially nonhomogeneous equilibrium in a reaction-diffusion system with distributed delay.
\newblock \emph{J. Differential Equations}, 250(6):2779--2806, 2011.

\bibitem{lou2006effects}
Y. Lou.
\newblock On the effects of migration and spatial heterogeneity on single and multiple species.
\newblock \emph{J. Differential Equations}, 223:400--426, 2006.


\bibitem{Yan2010}
X.-P. Yan and W.-T. Li.
\newblock Stability of bifurcating periodic solutions in a delayed reaction-diffusion population model.
\newblock \emph{Nonlinearity}, 23(6):1413--1431, 2010.

\bibitem{Yan2012}
X.-P. Yan and W.-T. Li.
\newblock Stability and Hopf bifurcations for a delayed diffusion system in population dynamics.
\newblock \emph{Discrete Contin. Dyn. Syst. Ser. B}, 17(1):367--399, 2012.


\bibitem{su2012hopf}
Y. Su, J. Wei, and J. Shi.
\newblock Hopf bifurcation in a diffusive logistic equation with mixed delayed and instantaneous density dependence.
\newblock \emph{J. Dynam. Differential Equations}, 24(4):897--925, 2012.

\bibitem{faria2001normal}
T. Faria.
\newblock Normal forms for semilinear functional differential equations in Banach spaces and applications. II.
\newblock \emph{Discrete Contin. Dyn. Syst.}, 7(1):155--176, 2001.


\bibitem{faria2002smoothness}
T. Faria, W. Huang, and J. Wu.
\newblock Smoothness of center manifolds for maps and formal adjoints for semilinear FDEs in general Banach spaces.
\newblock \emph{SIAM J. Math. Anal.}, 34(1):173--203, 2002.

\bibitem{hassard1981theory}
B. D. Hassard, N. D. Kazarinoff, and Y. Wan.
\newblock \emph{Theory and Applications of Hopf Bifurcation}.
\newblock Cambridge University Press, Cambridge, 1981.

\bibitem{gurney1980nicholson}
W. S. C. Gurney, S. P. Blythe, and R. M. Nisbet.
\newblock Nicholson's blowflies revisited.
\newblock \emph{Nature}, 287:17--21, 1980.

\bibitem{faria2002stability}
T. Faria and W. Huang.
\newblock Stability of periodic solutions arising from Hopf bifurcation for a reaction-diffusion equation with time delay.
\newblock In \emph{Differential Equations and Dynamical Systems}, volume 31 of Fields Inst. Commun., pages 125--141. Amer. Math. Soc., Providence, RI, 2002.

\bibitem {Cantrell2003}
R. S. Cantrell and C. Cosner.
\newblock \emph{Spatial Ecology via Reaction-Diffusion Equations}.
\newblock John Wiley and Sons, Chichester, 2003.

\bibitem{averill2017role} 
I. Averill, K.-Y. Lam, Y. Lou, 
\textit{The role of advection in a two-species competition model: a bifurcation approach}, 
Mem. Am. Math. Soc. \textbf{245} (1161) (2017), v+117.

\bibitem{lam2011concentration}
K.-Y. Lam, 
\textit{Concentration phenomena of a semilinear elliptic equation with large advection in an ecological model}, 
J. Differential Equations. \textbf{250} (1) (2011) 161--181.

\bibitem{lam2012limiting}
K.-Y. Lam, 
\textit{Limiting profiles of semilinear elliptic equations with large advection in population dynamics II}, 
SIAM J. Math. Anal. \textbf{44} (3) (2012) 1808--1830.


\bibitem{belgacem1995effects} 
F. Belgacem, C. Cosner, 
\textit{The effects of dispersal along environmental gradients on the dynamics of populations in heterogeneous environments}, 
Can. Appl. Math. Q. \textbf{3} (4) (1995) 379--397.

\bibitem{cantrell2003spatial} 
R.S. Cantrell, C. Cosner, \textit{Spatial Ecology via Reaction-Diffusion Equations}, 
Wiley Series in Mathematical and Computational Biology, John Wiley \& Sons, Ltd., 
Chichester, 2003.

\bibitem{cosner2014reaction}
C. Cosner, 
\textit{Reaction-diffusion-advection models for the effects and evolution of dispersal}, 
Discrete Contin. Dyn. Syst. \textbf{34} (5) (2014) 1701--1745.

\bibitem{lam2015evolution}
K.-Y. Lam, Y. Lou, F. Lutscher, 
\textit{Evolution of dispersal in closed advective environments}, 
J. Biol. Dyn. \textbf{9} (suppl. 1) (2015) 188--212.

\bibitem{lou2014evolution}
Y. Lou, F. Lutscher, 
\textit{Evolution of dispersal in open advective environments}, 
J. Math. Biol. \textbf{69} (6-7) (2014) 1319--1342.

\bibitem{lou2016qualitative}
Y. Lou, D. Xiao, P. Zhou, 
\textit{Qualitative analysis for a Lotka-Volterra competition system in advective homogeneous environment}, 
Discrete Contin. Dyn. Syst. \textbf{36} (2) (2016) 953--969.

\bibitem{lou2019global}
Y. Lou, X.-Q. Zhao, P. Zhou, 
\textit{Global dynamics of a Lotka-Volterra competition-diffusion-advection system in heterogeneous environments}, 
J. Math. Pures Appl. \textbf{9} (121) (2019) 47--82.

\bibitem{lou2015evolution}
Y. Lou, P. Zhou, 
\textit{Evolution of dispersal in advective homogeneous environment: the effect of boundary conditions}, 
J. Differential Equations \textbf{259} (1) (2015) 141--171.

\bibitem{vasilyeva2012how}
O. Vasilyeva, F. Lutscher, 
\textit{How flow speed alters competitive outcome in advective environments}, 
Bull. Math. Biol. \textbf{74} (12) (2012) 2935--2958.

\bibitem{Hess1991}
P. Hess,
\textit{Periodic-Parabolic Boundary Value Problems and Positivity},
Pitman Research Notes in Mathematics Series, vol. 247,
Longman Scientific \& Technical, Copublished in the United States with John Wiley \& Sons, Inc.,
Harlow/New York, 1991.

\bibitem{zhou2018evolution}
P. Zhou, X.-Q. Zhao, 
\textit{Evolution of passive movement in advective environments: general boundary condition}, 
J. Differential Equations \textbf{264} (6) (2018) 4176--4198.

\bibitem{gopalsamy1988time}
K. Gopalsamy, M. R. S. Kulenovi\'{c}, and G. Ladas, 
\textit{Time lags in a ``food-limited'' population model}, 
Appl. Anal. \textbf{31} (1988), 225--237.

\bibitem{gopalsamy1990environmental}
K. Gopalsamy, M. R. S. Kulenovi\'{c}, and G. Ladas, 
\textit{Environmental periodicity and time delays in a ``food-limited'' population model}, 
J. Math. Anal. Appl. \textbf{147} (1990), 545--555.

\bibitem{kuang1993delay}
Y. Kuang, 
\textit{Delay differential equations with applications in population dynamics}, 
Academic Press, New York, 1993.

\bibitem{allee1931animal}
W.C. Allee, 
\textit{Animal Aggregations: A Study in General Sociology}, 
University of Chicago Press, Chicago, 1931.

\bibitem{shi2006persistence}
J. Shi, R. Shivaji, 
\textit{Persistence in reaction diffusion models with weak Allee effect}, 
J. Math. Biol. \textbf{52} (6) (2006) 807--829.

\bibitem{stephens1999what}
P.A. Stephens, W.J. Sutherland, R.P. Freckleton, 
\textit{What is the Allee effect?}, 
Oikos \textbf{87} (1999) 185--190.

\bibitem{davidson2001effects}
F.S. Davidson, S.A. Gourley, 
\textit{The effects of temporal delays in a model for a food-limited, diffusing population}, 
J. Math. Anal. Appl. \textbf{261} (2001) 633--648.


\end{thebibliography}
\end{document}